\documentclass[10pt,reqno]{amsart}
\usepackage{amssymb,version,graphicx,fancybox,pifont}
\usepackage{multirow}
\usepackage{url,hyperref}
\usepackage{mathrsfs}
\usepackage{amsmath}
\usepackage{subfigure}
\usepackage{enumerate}
\usepackage{booktabs}
\usepackage[table, dvipsnames]{xcolor}
\usepackage{braket}
\usepackage{cleveref}
\usepackage{bm}
\catcode`\@=11
\@addtoreset{equation}{section}   

\@addtoreset{table}{section}   
\renewcommand\thefigure{\thesection.\@arabic\c@figure}
\@addtoreset{figure}{section}   
\renewcommand\thetable{\thesection.\@arabic\c@table}

 \newcommand{\new}{\newcommand*}
 \new{\rnew}{\renewcommand*}
  \new{\newe}{\newenvironment*}
 \new{\stl}{\setlength}
 \stl{\arraycolsep}{0.5mm}

\newtheorem{thm}{\bf Theorem}

\newenvironment{theorem}{\begin{thm}} {\end{thm}}
\newtheorem{lmm}{\bf Lemma}

\newenvironment{lemma}{\begin{lmm}}{\end{lmm}}
\theoremstyle{remark}

\theoremstyle{definition}
\newtheorem{defn}{Definition}[section]
\usepackage{multirow}

\newcommand {\bgeq}[1]{\begin{equation}\label{#1}}
\newcommand \edeq {\end{equation}}
\newcommand \bgth {\begin{theorem}\label}
\newcommand \edth {\end{theorem}}
\newcommand \bglm {\begin{lemma}\label}
\newcommand \edlm {\end{lemma}}
\newcommand {\bgar}[1]{\begin{array}{#1}}
\newcommand {\edar}{\end{array}}

\newtheorem{remark}{Remark}

\newcommand{\bs}[1]{\boldsymbol{#1}}
\usepackage{algorithm, algorithmic}

\title[C-PINN]{C-PINN: A neural network framework based on the Cord\`{e}s condition for solving  linear and fully nonlinear equations in non-divergence form and its applications}
\author[B. Hu, L. Jin and Z. Li]{}

\thanks{$^\dag$Corresponding author  email:  zxli@shnu.edu.cn. \\ \indent The work was partially supported by the NNSF of China (No.12271366, 12571391, 11871043).}

\begin{document} \maketitle
\centerline{Bingcheng Hu, Lixiang Jin, Zhaoxiang Li$^{\dag}$}
\centerline{Department of  Mathematics, Shanghai Normal University, Shanghai,
200234, P.R. China}


%

\begin{abstract}
In this paper, we propose a novel Physics-Informed Neural Network (PINN) framework based on the Cord\`{e}s condition for solving both linear and fully nonlinear partial differential equations (PDEs) in non-divergence form, together with their applications. By incorporating the operator structure into the loss function, the proposed method improves the conditioning of the associated optimization problem, thereby enhancing training stability and solution accuracy. The framework is further extended to include Hamilton-Jacobi-Bellman and Monge-Amp\`{e}re equations, with applications to optimal transport. Numerical experiments demonstrate the effectiveness and robustness of the method, as well as its capability to address high-dimensional problems, highlighting the promise of learning-based approaches for tackling challenging PDEs.
\end{abstract}

\section{Introduction}
Linear partial differential equations in non-divergence form, such as
\begin{equation} \label{eq:general_nondiv}
    \mathcal{L} u := A(\bs x) : D^2 u + \bs{b}(\bs x) \cdot \nabla u - c(\bs x) u = f(\bs x), \quad \bs x \in \Omega,
\end{equation}
where $D^2u$ denotes the Hessian matrix of $u$, $A: D^2 u$  represents the Frobenius inner product between the matrices $A$ and $D^2 u$, and $\Omega$ is a bounded convex domain in $\mathbb{R}^d$, have been extensively studied in the classical theory of elliptic equations \cite{Gilbarg2001,Krylov2008}. Such equations arise naturally in stochastic control and differential games, where the associated value functions satisfy Hamilton-Jacobi-Bellman (HJB) equations of the form
        \begin{equation} \label{eq:HJB}
            \sup_{\alpha \in \Lambda}(\mathcal{L}^{\alpha}u -f^{\alpha})=0 \quad \text{in }\Omega, \quad u = 0 \quad \text{on }\partial \Omega,
        \end{equation}
with $\Lambda$ denoting the admissible control set and $\mathcal{L}^{\alpha}$ a family of second-order differential operators indexed by $\alpha \in \Lambda$. These equations are typically fully nonlinear and belong to the class of non-divergence form PDEs \cite{Crandall1992,Dragoni2010,Krylov1980}. Another important example is the Monge-Amp\`{e}re (MA) equation,
\begin{equation} \label{eq1}
 {\rm det}(D^{2}u)=f \quad \text{in }\Omega, \quad u = g \quad \text{on }\partial \Omega,
\end{equation}
which serves as a prototypical fully nonlinear PDE involving the determinant of the Hessian. It is closely related to non-divergence form equations through linearization and iterative solution procedures \cite{Caffarelli1995,Dean2003}. The intrinsic non-variational structure of these problems poses significant challenges for both theoretical analysis and numerical approximation \cite{Feng2011,Nochetto2018}.

Despite their theoretical and practical importance, the numerical approximation of linear non-divergence and fully nonlinear PDEs remains highly difficult \cite{Barles1990,Feng2013}. The ability to compute solutions for these equations accurately is crucial for advanced applications, including financial engineering driven by HJB frameworks, as well as Wasserstein-based generative models \cite{Arjovsky2017}, domain adaptation \cite{Courty2017}, computer vision \cite{Solomon2015}, and resource allocation problems \cite{Peyre2019}. Classical approaches, including finite difference methods\cite{Froese2011,Oberman2006}, finite element methods \cite{Feng2019,Kawecki2021,Pal2025,Brenner2011} and spectral method \cite{Wpeipei}, require carefully designed discretizations to ensure stability and consistency. This often results in complicated implementations and limited flexibility \cite{Debrabant2013,Jensen2003,Krylov1997}. Moreover, strong nonlinearity, heterogeneous non-divergence operators, and intrinsic convexity constraints continue to present significant challenges for traditional numerical schemes \cite{Feng2011APNUM,Oberman2008}.

In recent years, Physics-Informed Neural Networks (PINNs) have emerged as a promising mesh-free alternative \cite{Raissi2019, Sirignano2018}. By exploiting the universal approximation capability of deep neural networks and employing automatic differentiation to evaluate differential operators, PINNs incorporate the governing equations directly into the loss function \cite{Karniadakis2021}. Subsequently, a variety of enhancements have been proposed to improve their performance, including adaptive loss weighting and sampling \cite{Xiang2022, Wang2024}, architectural modifications \cite{Jagtap2020,Cho2026}, training stabilization techniques \cite{Lin2024,Lu2021,Guo2022,He2024} and rigorous error analysis. Despite these advances, most existing approaches are primarily tailored to PDEs in divergence form or problems involving lower-order derivatives \cite{Cuomo2022}. Consequently, their applicability to non-divergence and fully nonlinear equations remains limited \cite{Darbon2020,Han2018}. In particular, direct minimization of PDE residual often leads to highly non-convex loss landscapes,  gradient instabilities, and poor convergence, primarily due to the presence of high-order, poorly conditioned derivatives \cite{Krishnapriyan2021,Wang2021,Wang2022}.

Efficient numerical methods for linear and fully nonlinear equations in non-divergence form, particularly  HJB and MA equations, are of fundamental importance in applied mathematics. The HJB equation serves as the cornerstone of stochastic optimal control, while the MA equation plays a central role in optimal transport theory\cite{Brenier1991,Santambrogio2015,Villani2003}. Despite their distinct physical origins, these equations share a profound structural connection. Through algebraic duality, the determinant-based MA equation \eqref{eq1} can be equivalently reformulated as an HJB equation by taking the infimum over symmetric positive-definite matrices
 \begin{equation}
 (\det D^2 u)^{1/d} = \inf_{A > 0,\ \det A = 1} \frac{1}{d}\mathrm{tr}(A D^2 u),
\end{equation}
as shown in \cite{Krylov1987}. Because of this equivalence, the numerical resolution of both equations intrinsically relies on handling complex linear non-divergence operators.

In this work, we exploit this equivalence and propose a novel PINN framework based on the Cord\`{e}s condition(C-PINN) to improve the loss formulation for non-divergence form PDEs. The Cord\`{e}s condition, originally developed in elliptic theory, provides sufficient criteria for well-posedness and regularity by controlling the anisotropy of the coefficient matrix and ensuring near-uniform ellipticity of the operator \cite{Cordes1956}. Motivated by this property, we incorporate the Cord\`{e}s condition into the learning framework. By exploiting the induced structural constraint, the proposed approach improves the conditioning of the optimization problem and leads to a smoother loss landscape. Furthermore, the framework is naturally extended to handle HJB and MA equations, thereby enabling efficient numerical treatment of optimal transport problems within a unified learning-based paradigm \cite{Ruthotto2020}.

The primary objective of this paper is to develop a Cord\`{e}s condition based Physics-Informed Neural Network framework for solving linear non-divergence form and fully nonlinear PDEs, with applications to HJB equations, MA equations, and optimal transport. The main contributions are summarized as follows:
\begin{itemize}
    \item We propose a novel PINN framework based on the Cord\`{e}s condition, which improves the loss formulation and enhances the stability and conditioning of the optimization process for non-divergence form PDEs.
    \item The proposed method is further extended to solve fully nonlinear equations, including HJB and MA equations, demonstrating its effectiveness in handling highly complex nonlinear problems.
    \item The framework is successfully applied to optimal transport problems, enabling accurate approximation of transport maps.
    \item Numerical experiments validate the accuracy, stability and robustness of the proposed method, while also confirming its effectiveness in handling high-dimensional problems.
\end{itemize}

The remainder of the paper is organized as follows. Section 2 introduces the C-PINN framework and analyzes its loss formulation. Section 3 presents training strategies. Section 4 provides numerical experiments for non-divergence equations and extensions to HJB and MA equations. Section 5 focuses on optimal transport applications. Section 6 concludes the paper.


\section{Architecture and methodology of C-PINN}

\subsection{Physics-Informed Neural Networks}
Physics-Informed Neural Networks are a kind of unsupervised learning framework, which can seamlessly integrate physical laws governed by partial differential equations into the learning process of neural networks \cite{Raissi2019}. Consider a general PDE defined on a bounded domain $\Omega \subset \mathbb{R}^d$ with boundary $\partial\Omega$:
\begin{equation}
    \begin{aligned}
        \mathcal{F}[u](\bs x) &= f(\bs x), \quad \bs x \in \Omega,\\
        \mathcal{B}[u](\bs x) &= g(\bs x), \quad \bs x \in \partial\Omega,
    \end{aligned}
\end{equation}
where $\mathcal{F}$ is a differential operator, $\mathcal{B}$ denotes the boundary condition operator, and $u$ is the exact solution. In the PINN framework, a deep neural network $u_\theta$ parameterized by weights and biases $\theta$ is employed to approximate the latent solution $u$. The required partial derivatives of $u_\theta$ with respect to the spatial coordinates are computed exactly and efficiently using automatic differentiation.

The neural network is trained by minimizing a composite loss function that penalizes both the data mismatch at the boundaries and the physical residual in the interior domain. The standard PINN loss function is formulated as:
\begin{equation}
    \boldsymbol{L}(\theta) = \lambda_1 \boldsymbol{L}_{\rm{pde}}(\theta) + \lambda_2 \boldsymbol{L}_{\rm{bc}}(\theta),
\end{equation}
where the PDE residual loss $\boldsymbol{L}_{\rm{pde}}(\theta)$ and the boundary loss $\boldsymbol{L}_{\rm{bc}}(\theta)$ are defined as mean squared errors (MSE) over a set of sampled collocation points:
\begin{equation}
    \begin{aligned}
        \boldsymbol{L}_{\rm{pde}}(\theta) &= \frac{1}{N_{\rm{pde}}} \sum_{i=1}^{N_{\rm{pde}}} |\mathcal{F}[u_\theta](\bs x_{\rm{pde}}^{(i)}) - f(\bs x_{\rm{pde}}^{(i)})|^2, \\
        \boldsymbol{L}_{\rm{bc}}(\theta) &= \frac{1}{N_{\rm{bc}}} \sum_{i=1}^{N_{\rm{bc}}} |\mathcal{B}[u_\theta](\bs x_{\rm{bc}}^{(i)}) - g(\bs x_{\rm{bc}}^{(i)})|^2.
    \end{aligned}
\end{equation}
Here, $\{\bs x_{\rm{pde}}^{(i)}\}_{i=1}^{N_{\rm{pde}}}$ and $\{\bs x_{\rm{bc}}^{(i)}\}_{i=1}^{N_{\rm{bc}}}$ represent the residual and boundary collocation points, respectively, and $\lambda_1, \lambda_2$ are the penalty weights used to balance the interplay between different loss terms.

While standard PINNs have achieved remarkable success in solving various differential equations, directly applying this framework to non-divergence form linear equations or fully nonlinear equations often leads to severe optimization challenges. The complex nonlinearities and high-order derivatives can cause gradient pathologies during backpropagation, resulting in poor convergence and low accuracy. This intrinsic limitation motivates the introduction of the Cord\`{e}s condition to reformulate the PDE residual, which constructs the foundation of our proposed C-PINN method.

\subsection{Cord\`{e}s Condition}
In this section, we introduce the Cord\`{e}s condition, a fundamental structural assumption for non-divergence form elliptic operators. This condition plays a pivotal role in reformulating the original partial differential equation into a strictly coercive form, which subsequently provides the mathematical foundation for our proposed C-PINN framework.

For Eq. \eqref{eq:general_nondiv}, the well-posedness and the spectral properties of the operator $\mathcal{L}$ are predominantly governed by its principal part, $A: D^2 u$. The classical Cord\`{e}s condition imposes a specific constraint on the eigenvalues of this coefficient matrix.

\begin{defn}[Cord\`{e}s Condition]\label{def:cordes_generalized}
    The operator $\mathcal{L}$ is said to satisfy the Cord\`{e}s condition if one of the following two cases holds for some constant $\varepsilon \in (0, 1)$:
    \begin{itemize}
        \item whenever $\bs{b} \equiv \bs{0}$ and $c \equiv 0$, then for almost every $\bs x \in \Omega$,
            \begin{equation} \label{eq:cordes1}
                \frac{\|A(\bs x)\|_F^2}{(\text{tr}(A(\bs x)))^2} \le \frac{1}{d - 1 + \varepsilon}.
            \end{equation}
        \item whenever $\bs{b} \not\equiv \bs{0}$ or $c \not\equiv 0$, there exists a parameter $\lambda > 0$ such that for almost every $\bs x \in \Omega$,
            \begin{equation} \label{eq:cordes2}
                \frac{\|A(\bs x)\|_F^2 + {|\bs{b}(\bs x)|^2}/{2\lambda} + {c^2(\bs x)}/{\lambda^2}}{\left(\text{tr}(A(\bs x)) + {c(\bs x)}/{\lambda}\right)^2} \le \frac{1}{d + \varepsilon},
            \end{equation}
    \end{itemize}
    where $\|A\|_F$ and $\text{tr}(A)$ denote the Frobenius norm and the trace of the matrix $A$, respectively.
\end{defn}
The inequality in Eq. \eqref{eq:cordes1} and Eq. \eqref{eq:cordes2} essentially restricts the dispersion of the eigenvalues of $A(\bs x)$, ensuring that the principal part of the operator does not deviate excessively from the standard Laplacian operator.

The profound theoretical implication of the Cord\`{e}s condition lies in its ability to stably reformulate the PDE. By defining a spatially varying scaling function:
\begin{equation} \label{eq:scaling_function}
    \alpha(\bs x) = \frac{\text{tr}(A(\bs x))}{\|A(\bs x)\|_F^2},
\end{equation}
one can multiply both sides of Eq. \eqref{eq:general_nondiv} by $\alpha(\bs x)$ to obtain the scaled equation $\alpha(\bs x) \mathcal{L} u = \alpha(\bs x) f(\bs x)$. Under the Cord\`{e}s condition, classical PDE theory guarantees that the scaled principal operator $\alpha(\bs x) A(\bs x) : D^2 u$ becomes a small perturbation of the Laplacian. Consequently, the reformulated operator exhibits strong coercivity in the Sobolev space $H^2(\Omega)$.

\begin{remark}
    Although a generalized Cord\`{e}s condition exists for operators with lower-order terms by introducing an auxiliary parameter \cite{Smears2013}, in the practical implementation of C-PINN, we compute the scaling function $\alpha(\bs x)$ relying exclusively on the principal part $A(\bs x)$ \cite{Gilbarg2001}. Mathematically, the lower-order terms represent compact perturbations, and the extreme non-convexity of the optimization landscape is overwhelmingly dominated by the second-order derivatives. By isolating and scaling the highest-order terms, our framework not only restores the strict convexity of the residual functional but also strictly restrains the computational complexity. This targeted preconditioning design circumvents the prohibitive computational overhead of tuning and calculating generalized multipliers for lower-order perturbations across all collocation points, yielding an elegant balance between theoretical rigor and computational scalability.
\end{remark}

From the perspective of deep learning and optimization, directly minimizing the naive residual of Eq.~\eqref{eq:general_nondiv} often yields a highly non-convex loss landscape accompanied by severe gradient pathologies \cite{Krishnapriyan2021,Wang2021}. By contrast, introducing the scaling function $\alpha(\bs x)$ derived from the Cord\`{e}s condition transforms the PDE residual into a functional that is better-conditioned with respect to the highest-order derivatives. This transformation effectively acts as a contraction mapping during training, smoothing the optimization trajectory and enabling stable and accurate approximation of the latent solution.

\subsection{Cord\`{e}s Loss for Linear Non-divergence Equation}
Building upon the theoretical foundation of the Cord\`{e}s condition established in the previous section, we now formulate the modified loss function for training the C-PINN. The primary objective is to embed the properly scaled PDE residual into the neural network's optimization objective, thereby mitigating the gradient pathologies caused by the unconditioned highest-order derivatives.

As established in classical elliptic PDE theory, in Eq. \eqref{eq:general_nondiv}, the well-posedness and regularity of the operator are strictly dominated by its principal second-order part $A(\bs x) : D^2 u$. The lower-order terms merely represent compact perturbations in the Sobolev space $H^2(\Omega)$. From an optimization perspective within the PINN framework, it is precisely this principal part that induces severe non-convexity and gradient pathologies. Therefore, the core strategy to stabilize the training trajectory is to properly precondition the highest-order terms. By leveraging the Cord\`{e}s condition, we can introduce a spatially varying multiplier that forces the ill-conditioned principal operator to behave analogously to the standard Laplacian $\Delta$. This fundamental transformation mathematically guarantees a strict contraction mapping, as formalized in the following theorem.

\begin{lemma} [Cord\`{e}s \cite{Neilan2017}] \label{thm:cordes}
    Let $\Omega \subset \mathbb{R}^n$ be a bounded domain with a sufficiently smooth boundary. Assume the uniformly elliptic coefficient matrix $A(\bs x)$ satisfies the $n$-dimensional Cord\`{e}s condition with parameter $\varepsilon \in (0, 1]$. By selecting the optimal point-wise multiplier $\lambda(\bs x) = {\text{tr}(A)}/{\text{tr}(A^2)}$, the preconditioned residual exhibits a strict contraction mapping for any $u \in H^2(\Omega) \cap H_0^1(\Omega)$:
    \begin{equation} \label{eq:cordes_loss}
        \|\Delta u - \lambda A : D^2 u\|_{L^2(\Omega)} \le \sqrt{1 - \varepsilon} \|\Delta u \|_{L^2(\Omega)}.
    \end{equation}
\end{lemma}

\begin{theorem}\label{thm:nn_cordes}
    Assume the neural network architecture $u_\theta \in C^2(\Omega)$, and satisfies the homogeneous Dirichlet boundary condition exactly via a hard-constraint ansatz $u_\theta(\bs x) = \mathcal{D}(\bs x)\mathcal{N}_\theta(\bs x)$ where $\mathcal{D}(\bs x) = 0$ on $\partial\Omega$. If the coefficient matrix $A(\bs x)$ satisfies the assumptions in Lemma \ref{thm:cordes}, then the preconditioned neural network approximation satisfies the strict contraction mapping:
    \begin{equation} \label{eq:nn_cordes_loss}
        \|\Delta u_\theta - \lambda A : D^2 u_\theta\|_{L^2(\Omega)} \le \sqrt{1 - \varepsilon} \|\Delta u_\theta \|_{L^2(\Omega)}.
    \end{equation}
\end{theorem}

\begin{proof}
    To establish the contraction inequality for $u_\theta$, it is sufficient to prove that the parameterized function $u_\theta(\boldsymbol{x})$ inherently resides within the Sobolev space $H^2(\Omega) \cap H_0^1(\Omega)$ required by Lemma \ref{thm:cordes}.

    First, we examine the regularity of the standard feed-forward neural network $\mathcal{N}_\theta: \mathbb{R}^n \to \mathbb{R}$. The network is composed of successive affine transformations and non-linear activations. We employ the SiLU activation function, defined as
    $$ \sigma(z) = z \cdot \varsigma(z) = \frac{z}{1 + e^{-z}}. $$
    Since the logistic sigmoid function $\varsigma(z)$ is infinitely differentiable on $\mathbb{R}$, it follows directly that $\sigma(z) \in C^\infty(\mathbb{R})$. Because finite compositions and linear combinations of $C^\infty$ functions preserve infinite differentiability, the global neural network output inherently satisfies $\mathcal{N}_\theta \in C^\infty(\bar{\Omega})$.

    Next, we consider the hard-constraint ansatz $u_\theta(\boldsymbol{x}) = \mathcal{D}(\boldsymbol{x}) \mathcal{N}_\theta(\boldsymbol{x})$. Given the assumption that the distance function is sufficiently smooth, $\mathcal{D} \in C^k(\bar{\Omega})$ with $k \ge 2$, the product rule of classical derivatives guarantees that the composite approximation $u_\theta \in C^2(\bar{\Omega})$.

    For a bounded domain $\Omega$ with a Lipschitz continuous boundary, we have the standard continuous embedding $C^2(\bar{\Omega}) \subset H^2(\Omega)$, which implies that all up to second-order weak derivatives of $u_\theta$ are square-integrable.

    Furthermore, by construction, $\mathcal{D}(\boldsymbol{x}) = 0$ for all $\boldsymbol{x} \in \partial\Omega$. Therefore, the trace of $u_\theta$ on the boundary vanishes:
    $$ \gamma_0 u_\theta = u_\theta \big|_{\partial\Omega} = \left( \mathcal{D} \cdot \mathcal{N}_\theta \right) \big|_{\partial\Omega} = 0. $$
    By the trace theorem, functions in $H^1(\Omega)$ with zero trace uniquely belong to the subspace $H_0^1(\Omega)$.

    Combining these regularity results yields $u_\theta \in H^2(\Omega) \cap H_0^1(\Omega)$. Since Lemma \ref{thm:cordes} holds universally for any test function $u \in H^2(\Omega) \cap H_0^1(\Omega)$, substituting $u$ with the neural network approximation $u_\theta$ validates the inequality \eqref{eq:nn_cordes_loss}.
\end{proof}

\begin{remark}
    Theorem \ref{thm:nn_cordes} mathematically justifies the integration of the Cord\`{e}s multiplier $\lambda(\bs x)$ into the PINN loss function. It guarantees that the preconditioned linear operator $\lambda A : D^2$ acts as a strict contraction towards the Laplacian operator $\Delta$ uniformly across the parameterized function space, thereby effectively bounding the spectrum of the deeply coupled nonlinear loss landscape and preventing gradient pathologies during backpropagation.
\end{remark}

The theoretical guarantee of this strict contraction mapping serves as the cornerstone for our proposed C-PINN. By embedding the optimal multiplier $\lambda(\bs x)$ directly into the neural network's empirical risk minimization formulation, we fundamentally reshape the geometry of the optimization landscape. Rather than navigating the steep cliffs and singular valleys characteristic of standard PINNs, the network now descends through a well-conditioned, strictly shrinking basin. This mathematically derived regularization intrinsically prevents gradient explosions and ensures robust convergence, establishing a rigorous bridge between classical functional analysis and deep learning optimization.

Hence, for the linear non-divergence equations, we can define the Cord\`{e}s loss as is followed:
\begin{equation} \label{eq:closs_for_nondiv}
    \boldsymbol{L}_{\text{Cord\`{e}s}}(\theta) = \frac{1}{N_{\text{int}}} \sum_{i=1}^{N_{\text{int}}} \left| \frac{\text{tr}(A)}{\text{tr}(A^2)+\delta} \Big( \mathcal{L}u_{\theta}(\bs x_i) - f(\bs x_i) \Big) \right|^2,
\end{equation}
where $\{ \bs x_i\} _ {i=1} ^ {N_\text{int}}$ are the sampled points in the domain $\Omega$, and $N_\text{int}$ is the number of interior sampled points, and $\delta > 0$ is a small threshold ensuring numerical stability.

\section{Dual-Loop C-PINN for Nonlinear PDEs: A Unified Newton Framework}
While the static C-PINN framework demonstrates unconditional stability for linear non-divergence equations, extending this success to nonlinear PDEs, such as the HJB and MA equations, poses a fundamental theoretical challenge. Directly minimizing highly non-convex or non-differentiable operators often leads to gradient explosion in deep learning optimizers. To bridge this gap, we introduce a successive linearization paradigm. By employing an outer-loop iterative scheme to locally project the fully nonlinear operator into a sequence of surrogate linear PDEs, we seamlessly inject the Cord\`{e}s condition into each step, thereby unifying the treatment of diverse fully nonlinear equations under a single robust neural architecture.

\subsection{Solving the Hamilton-Jacobi-Bellman Equation via Semi-smooth Newton Iteration}
The HJB equations stands as the fundamental governing equation in stochastic optimal control. Its primary challenge stems from the highly non-convex and non-differentiable $\sup$ or $\inf$ operators over the control space, which invalidate classical smooth linearization techniques. We consider the general stationary HJB equation as Eq. \eqref{eq:HJB}, where the non-differentiable nature of the $\sup$ operator necessitates the use of the Semi-smooth Newton method.

In the $k$-th outer loop, given the frozen state $u^{(k)}$, we evaluate the Clarke generalized Jacobian of the $\max$ operator. This analytically corresponds to identifying the active control branch
\begin{equation}
    \alpha^* = \arg\max_{\alpha \in \varLambda} \left( \mathcal{L}^\alpha u^{(k)} - f^\alpha \right).
\end{equation}
By selecting this active policy $\alpha^*$, the generalized derivative is precisely $\mathcal{L}^{\alpha^*}$. The Semismooth Newton step is then formulated as
\begin{equation}
    \Big( \mathcal{L}^{\alpha^*} u^{(k)} - f^{\alpha^*} \Big) + \mathcal{L}^{\alpha^*} \Big( u^{(k+1)} - u^{(k)} \Big) = 0.
\end{equation}
Then, we notice that the previous state $\mathcal{L}^{\alpha^*} u^{(k)}$ is perfectly canceled, projecting the fully nonlinear HJB equation into a surrogate linear non-divergence form equation for the inner-loop neural solver
\begin{equation}
    A^{\alpha^*}(\bs x) : D^2u^{(k+1)} + \mathbf{b}^{\alpha^*}(\bs x) \cdot \nabla u^{(k+1)} - c^{\alpha^*}  u^{(k+1)} = f^{\alpha^*}(\bs x).
\end{equation}
To unconditionally stabilize the subsequent neural network training, we construct the static Cord\`{e}s preconditioner exclusively based on the frozen active diffusion matrix $A^{\alpha^*}$.

Ultimately, we obtain the Cord\`{e}s loss function for HJB equations in Semi-smooth Newton iteration:
\begin{equation}
    \boldsymbol{L}_{\text{Cord\`{e}s}}^{(k)}(\theta) = \frac{1}{N_{\text{int}}} \sum_{i=1}^{N_{\text{int}}} \left| \frac{\text{tr}(A^{(k)})}{\text{tr}((A^{(k)})^2) + \delta} \left( \mathcal{L}^{\alpha^*} u_\theta ^ {(k)}(\bs x_i) - f^{\alpha^*}(\bs x_i) \right) \right|^2.
\end{equation}

\subsection{Solving the Monge-Amp\`{e}re Equation via Smooth Newton-Picard Linearization}
While the HJB equation addresses non-smooth singularities, MA equation represents the pinnacle of geometric fully nonlinear PDEs, driven by the smooth determinant operator. When the differential operator is strictly differentiable, our generalized outer-loop framework elegantly degenerates into the classical smooth Newton-Picard linearization.

We consider the standard MA equation $F(D^2u) = \det(D^2u) - f = 0$. In the $k$-th outer loop, instead of selecting an active branch via generalized Jacobians, we compute the exact Fr\'echet derivative of the determinant operator at the frozen state $u^{(k)}$. According to Jacobi's formula, the derivative of a determinant is given by the double-dot product with its cofactor matrix. Thus, the Newton expansion yields
\begin{equation}
    \det(D^2u^{(k)}) - f + \text{cof}(D^2u^{(k)}) : \Big( D^2u^{(k+1)} - D^2u^{(k)} \Big) = 0.
\end{equation}
By shifting the known quantities to the right-hand side, we obtain the surrogate linear non-divergence equation
\begin{equation}
    A^{(k)} : D^2u^{(k+1)} = \tilde{f}^{(k)},
\end{equation}
where the frozen state-dependent coefficient matrix is defined as $A^{(k)} = \text{cof}(D^2u^{(k)})$, and the linearized equivalent source term is $\tilde{f}^{(k)} = f - \det(D^2u^{(k)}) + A^{(k)} : D^2u^{(k)}$.

Consistent with the mechanism in Section 3.1, the inner-loop neural solver optimizes this static surrogate equation. We obtain the Cord\`{e}s loss function for MA equations in smooth Newton-Picard iteration:
\begin{equation}
    \boldsymbol{L}_{\text{Cord\`{e}s}}^{(k)}(\theta) = \frac{1}{N_{\text{int}}} \sum_{i=1}^{N_{\text{int}}} \left| \frac{\text{tr}(A^{(k)})}{\text{tr}((A^{(k)})^2) + \delta} \left( A^{(k)}(\bs x_i):D^2 u_\theta(\bs x_i) - \tilde{f}^{(k)}(\bs x_i) \right) \right|^2.
\end{equation}

This architectural duality underscores a profound theoretical unification: whether handling the non-smooth combinatorial switching in HJB or the smooth geometric curvature in MA, the proposed dual-loop framework consistently reduces the fully nonlinear operator into a surrogate linear PDE, which is subsequently resolved by the unconditionally stable C-PINN.

Furthermore, the loss landscape derived from the HJB equation is visualized to illustrate that the proposed loss function yields a smoother and more well-conditioned optimization landscape compared to the original formulation, as shown in Fig. \ref{fig:loss landscape}. Specifically, the landscapes are plotted by perturbing the network parameters along two random, filter-normalized directions around the converged local minimum. This visualization confirms that the Cord\`{e}s preconditioner structurally eradicates extreme local curvatures and widens the basin of attraction, which fundamentally prevents the optimizer from being trapped in pathological ravines during the entire training trajectory.
\begin{figure}[H]
    \centering
    \includegraphics[width=0.55\linewidth]{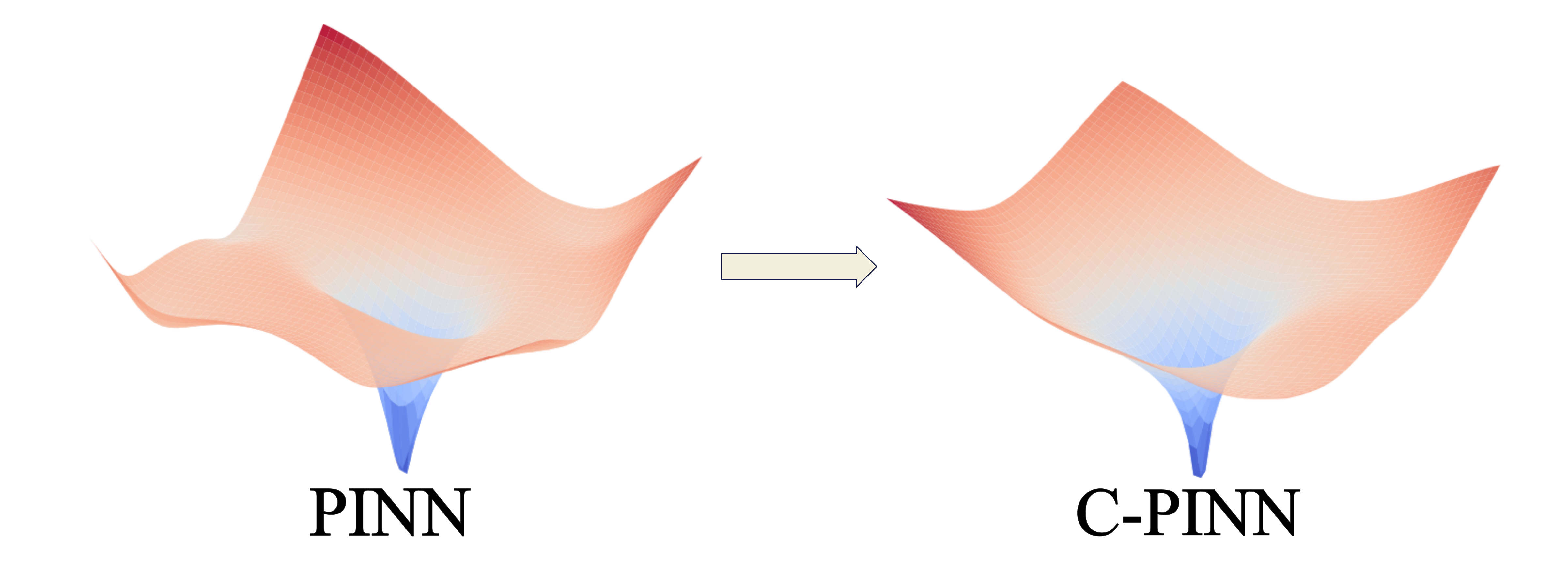}
    \vspace{-0.2in}
    \caption{The loss landscape of the HJB equation task}
    \label{fig:loss landscape}
\end{figure}
\vspace{-0.1in}
With the theoretical framework and the modified objective function of C-PINN now fully established, we proceed to evaluate its effectiveness, accuracy, and robustness through a series of comprehensive numerical experiments in the following section.

A comprehensive visual summary of the proposed framework is illustrated in Fig. \ref{fig:Dual-Loop C-PINN Framework}. The flowchart delineates the transition from the global warm-up phase to the subsequent dual-loop iteration, highlighting the architectural symmetry in handling HJB and MA equations.
\vspace{-0.1in}
\begin{figure}[H]
    \centering
    \includegraphics[width=0.9\linewidth]{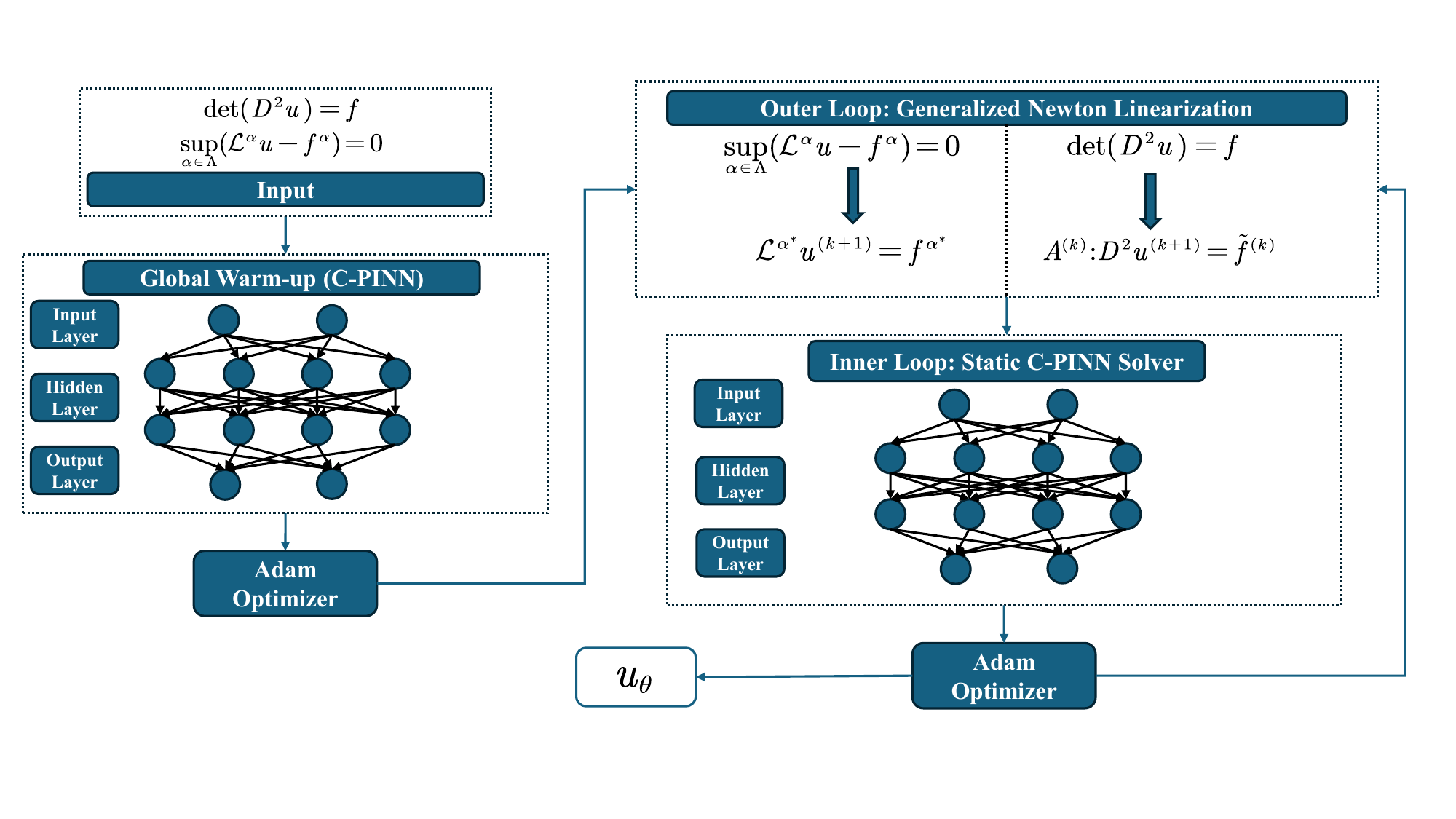}
    \vspace{-0.5in}
    \caption{Dual-Loop C-PINN framework}
    \label{fig:Dual-Loop C-PINN Framework}
\end{figure}

\begin{remark}
We acknowledge that providing a rigorous mathematical proof for the global convergence and stability of the coupled Adam-Newton optimization within the highly non-convex parameter space of deep neural networks remains a formidable open challenge in the current SciML literature. Nevertheless, the efficacy of the proposed scheme is strongly supported by empirical evidence. As will be demonstrated through the numerical experiments, the inclusion of the Cord\`{e}s multiplier plays a pivotal role in the optimization process. Comparative analyses reveal that our C-PINN strategy exhibits a significantly accelerated convergence rate and enhanced stability when evaluated against the PINN lacking the Cord\`{e}s preconditioning with Newton iteration. This underscores the practical necessity of transforming the non-smooth generalized Jacobian into a well-conditioned operator via the Cord\`{e}s multiplier.
\end{remark}

\section{Numerical experiments}
\subsection{Training Strategy and Optimization Metrics}
In this section, we present numerical experiments to verify the effect of the C-PINN by solving various linear non-divergence equations,  Hamilton-Jacobi-Bellman equations and Monge-Amp\`{e}re equations. Additionally, we compare the performance of C-PINN and PINN in these tasks.

Instead of employing a conventional composite loss with empirical penalty weights, we enforce the boundary conditions exactly by construction through a hard-constraint ansatz
\begin{equation}
    u_\theta(\bs{x}) = \mathcal{D}(\bs{x})\mathcal{N}_\theta(\bs{x}) = 0, \quad \bs{x} \in \partial \Omega,
\end{equation}
where $\mathcal{D}(\boldsymbol{x})$ denotes a smooth distance function designed to vanish identically on the boundary $\partial\Omega$, thereby enforcing the Dirichlet constraints exactly by construction. For the sake of brevity, the explicit analytical expressions of $\mathcal{D}(\boldsymbol{x})$ for all specific computational domains investigated in this work are systematically detailed in appendix. Consequently, the boundary loss term is completely eliminated from the optimization objective. The training process is thus reduced to solely minimizing the preconditioned interior residual
\begin{equation} \label{eq:loss}
    \boldsymbol{L}(\theta) = \boldsymbol{L}_{\text{Cord\`{e}s}}(\theta),
\end{equation}
where $\boldsymbol{L}_{\text{Cord\`{e}s}}(\theta)$ represents the Cord\`{e}s loss defined by the above sections. By employing this mathematically grounded preconditioning mechanism, the gradient dynamics are inherently stabilized during the optimization process. To minimize Eq. \eqref{eq:loss}, we use Adam optimization in \cite{Kingma2014} with learning rate $\eta = 3 \times 10^{-4}$.

To evaluate the accuracy, we compute two error metrics: the $l^\infty$ error, defined as the largest absolute difference between the numerical and exact solutions across all evaluating points, and the $l^2$ error, defined as the root mean square of the absolute differences across all points. The corresponding formulas are
\begin{equation}
    \begin{aligned}
        \| u - u_{\theta} \|_{l^2} &= \left( \frac{1}{N} \sum_{i=1}^N | u(\bs x_i) - u_{\theta}(\bs x_i) |^2 \right)^{1/2},\\
        \| u - u_{\theta} \|_{l^\infty} &= \sup_{1 \leq i \leq N} |u(\bs x_i) - u_{\theta}(\bs x_i)|,
    \end{aligned}
\end{equation}
where $u_\theta$ denotes the numerical solution obtained by the neural network and $u$ represents the exact solution.

To rigorously and quantitatively analyze the smoothing effect of the Cord\`{e}s preconditioner on the highly non-convex functional landscape, we introduce two dynamical metrics evaluated along the training trajectory: the global gradient norm $\|\nabla_\theta \boldsymbol{L}\|_2$ and the trajectory sharpness $\sigma_{\max}$ Proxy.

Specifically, the global gradient norm represents the first-order derivative information, reflecting the slope or velocity of the optimization path. It evaluates the $L^2$ norm of the loss gradient with respect to all trainable network parameters $\theta$ at iteration $t$. Geometrically, a smoothly decaying $\|\nabla_\theta \boldsymbol{L}\|_2$ indicates that the optimizer is stably converging towards a flat stationary point, whereas large and chaotic oscillations imply that the gradient flow is trapped in pathological ravines.

Furthermore, to evaluate the second-order geometric properties, we introduce the $\sigma_{\max}$ Proxy to quantify the curvature or inherent stiffness of the local landscape \cite{Cohen2021}. Since computing the exact Hessian matrix of the PDE residual is computationally prohibitive, we follow the Edge of Stability theory and adopt the local Lipschitz constant along the optimization path as an efficient proxy for the maximum singular value of the Hessian. It is defined as
\begin{equation}
    \sigma_{\max} \text{Proxy} = \frac{\|\nabla_{\theta} \boldsymbol{L}(\theta_t) - \nabla_{\theta} \boldsymbol{L}(\theta_{t-1})\|_2}{\|\theta_t - \theta_{t-1}\|_2 }.
\end{equation}
A larger $\sigma_{\max}$ Proxy implies an extremely sharp and irregular functional topology, which forces the optimization to become highly unstable or even diverge. Conversely, a smaller and bounded value demonstrates that the landscape has been effectively smoothed, allowing the optimizer to descend safely.

All experiments were implemented within a consistent training and evaluation framework, and the computations were carried out on a workstation with an NVIDIA RTX4090(24GB), Intel(R) Xeon(R) Gold 6430 CPU, using Python3.10 with PyTorch2.3.0 and CUDA12.4.1 on Ubuntu22.04.

\subsection{Example 4.1 Diffusion-Dominated Elliptic Equation in Non-divergence Form}
We consider the non-divergence equation:
\begin{equation}\label{eq:ex1}
    \begin{cases}
      \mathcal{L} u = -A:D^2 u =f  , \ & \text{in} \ \Omega, \\
      u = 0  , \ & \text{on} \ \partial \Omega,
    \end{cases}
\end{equation}
where $\Omega = (-2,2) \times (-2,2) $, and
\begin{equation}
    A\left( x_1,x_2 \right) =\left( \begin{matrix}
        \left( 2x_1-x_2 \right) ^{1/3}+4e^{2-x_1}&		\frac{1}{2}\sin \left( 10x_1x_2 \right) -\frac{1}{2}\left( x_1+2 \right) ^{1/2}\\
        \frac{1}{2}\sin \left( 10x_1x_2 \right) -\frac{1}{2}\left( x_1+2 \right) ^{1/2}&		|x_2-2x_1|^{1/4}+3\\
  \end{matrix} \right).
\end{equation}
\indent We begin with the case where the source term $f$ is chosen to match the smooth solution $u\left( x_1,x_2 \right) =\frac{1}{6}|x_1|^3\cos \left( x_2 \right)$. In this experiment, we randomly sample 10,000 interior points and 1,000 boundary points, and the solution is evaluated on a uniform $200 \times 200$ grid in the domain. Unless otherwise specified, the same experimental setup is adopted for all experiments.\\
\indent In the following, rather than comparing errors at fixed iterations, we perform a rigorous time-to-solution analysis. Table \ref{tab:ex1} details the computational cost (computed every 10 epochs) required for both methods to achieve specific $l^2$ error milestones for the smooth case, with the resulting C-PINN solution visualized in Fig. \ref{fig:ex1}. It is clearly evident that C-PINN dramatically outperforms the standard PINN in computational efficiency, easily penetrating high-accuracy milestones where the standard PINN stagnates. This superiority in convergence speed and ultimate accuracy is physically corroborated by the optimization dynamics shown in Fig. \ref{fig:gn1}, which demonstrates that C-PINN significantly enhances landscape smoothness and gradient stability.

\vspace{-0.05in}

\begin{table}[htbp]
    \centering
    \caption{Computational cost and actual errors for C-PINN and PINN to reach target $l^2$ accuracy milestones for the smooth case of \eqref{eq:ex1}.}
    \label{tab:ex1}
    \begin{tabular}{c l | r r c c}
    \toprule
    \multirow{2}{*}{\shortstack{Target \\ $l^2$ Error}} & \multirow{2}{*}{Method} & \multicolumn{2}{c}{Computational Cost} & \multicolumn{2}{c}{Actual Errors} \\
    \cmidrule(lr){3-4} \cmidrule(lr){5-6}

    & & Epochs & Time (s) & \makebox[2.2cm][c]{$l^2$ Error} & \makebox[2.2cm][c]{$l^ \infty$ Error} \\
    \midrule

    \multirow{2}{*}{1.0e-03}
    & C-PINN & \cellcolor{blue!10}80 & \cellcolor{blue!10}1.27 & 6.60e-04 & 1.74e-03 \\
    & PINN   & 300 & 5.02 & 9.78e-04 & 2.59e-03 \\
    \midrule

    \multirow{2}{*}{1.0e-04}
    & C-PINN & \cellcolor{blue!10}210 & \cellcolor{blue!10}3.34 & 9.80e-05 & 2.48e-04 \\
    & PINN   & 1590 & 25.93 & 9.98e-05 & 2.73e-04 \\
    \midrule

    \multirow{2}{*}{1.0e-05}
    & C-PINN & \cellcolor{blue!10}1620 & \cellcolor{blue!10}25.53 & 9.84e-06 & 3.21e-05 \\
    & PINN   & 7710 & 126.09 & 9.96e-06 & 2.85e-05 \\
    \midrule

    \multirow{2}{*}{1.0e-06}
    & C-PINN & \cellcolor{blue!10}6300 & \cellcolor{blue!10}99.39 & 9.95e-07 & 2.91e-06\\
    & PINN   & \cellcolor{gray!10}- & \cellcolor{gray!10}- & (2.92e-05) & (1.40e-05) \\

    \bottomrule
    \multicolumn{6}{l}{\small \textit{Note:} "-" indicates the method failed to reach the target accuracy within 40,000 epochs.} \\
    \multicolumn{6}{l}{\small \textit{\phantom{Note: }} Values in parentheses represent the final errors at the maximum epochs.} \\

    \end{tabular}
\end{table}

\vspace{-0.2in}
\begin{figure}[H]
    \centering
    \includegraphics[width=0.8\linewidth]{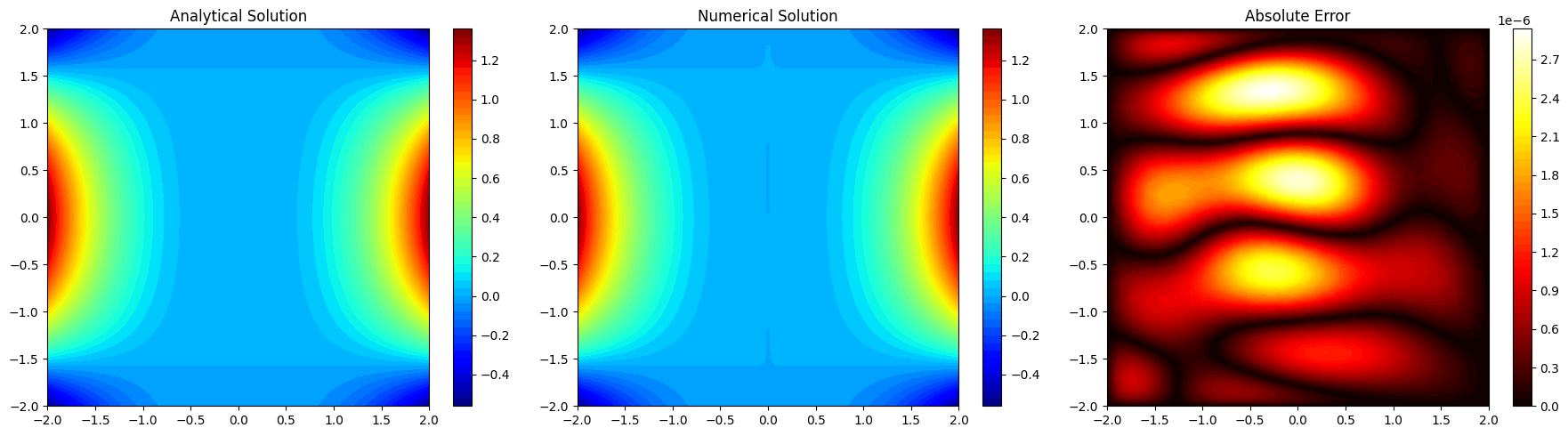}
    \caption{Results for the smooth case of \eqref{eq:ex1}. From left to right: exact solution, numerical solution, and absolute error.}
    \label{fig:ex1}
\end{figure}
\vspace{-0.2in}
\begin{figure}[H]
    \centering
    \includegraphics[width=0.8\linewidth]{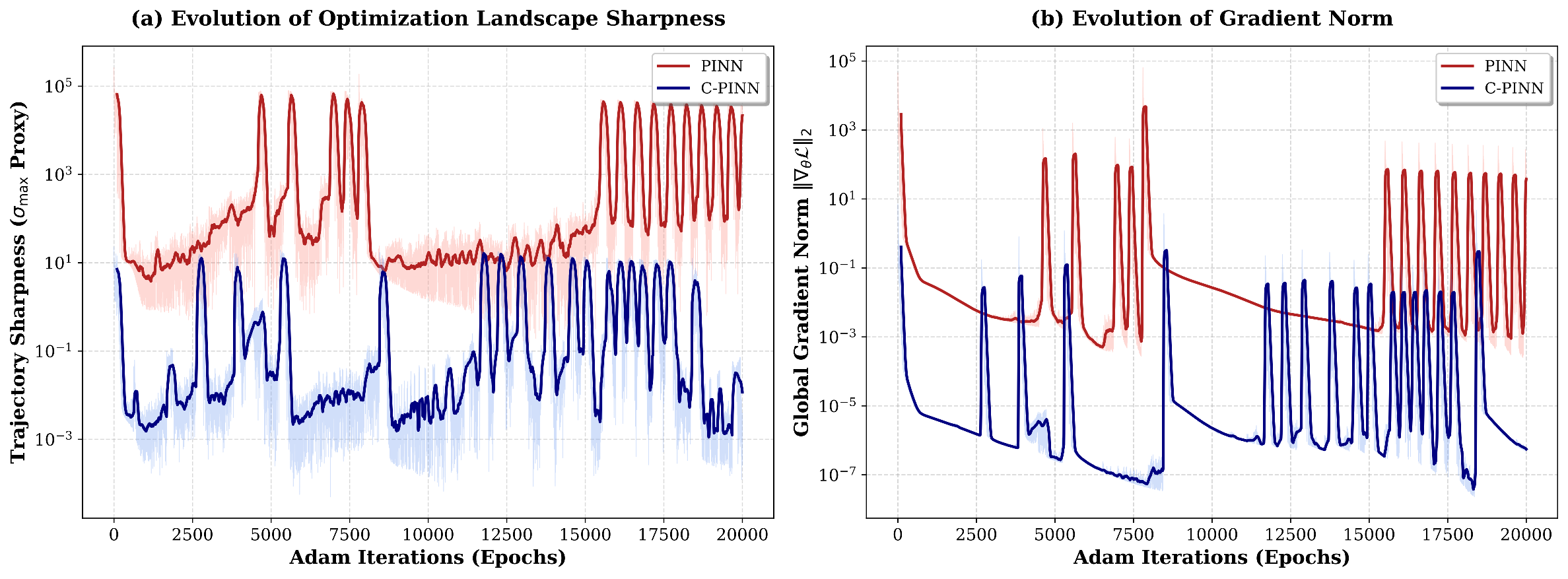}
    \caption{Optimization dynamics comparison: C-PINN vs. PINN for the smooth case of \eqref{eq:ex1}.}
    \label{fig:gn1}
\end{figure}

After that, we consider the weakly singular case with circular boundary domain. The source term $f$ is chosen to match the exact solution $u\left( x_1,x_2 \right) = (x_1 - x_2)^{8/3}$, where $\Omega = B_2(0)$. We present the time-to-solution analysis of the solutions obtained using  C-PINN and PINN, together with the computational cost in Table \ref{tab:ex1_2} for this task and the C-PINN solution in Fig \ref{fig:ex1_2}. This accuracy improvement is further corroborated by the optimization dynamics in Fig \ref{fig:gn2}, which demonstrates the significantly enhanced landscape smoothness and gradient stability of C-PINN. This specific test case is meticulously designed to corroborate that the proposed C-PINN framework remains robust and highly accurate even in the presence of weak singularities and non-rectangular curved boundaries.

\begin{table}[htbp]
    \centering
    \caption{Computational cost and actual errors for C-PINN and PINN to reach target $l^2$ accuracy milestones for the weakly singular case of \eqref{eq:ex1}.}
    \label{tab:ex1_2}
    \begin{tabular}{c l | r r c c}
    \toprule
    \multirow{2}{*}{\shortstack{Target \\ $l^2$ Error}} & \multirow{2}{*}{Method} & \multicolumn{2}{c}{Computational Cost} & \multicolumn{2}{c}{Actual Errors} \\
    \cmidrule(lr){3-4} \cmidrule(lr){5-6}

    & & Epochs & Time (s) & \makebox[2.2cm][c]{$l^2$ Error} & \makebox[2.2cm][c]{$l^ \infty$ Error} \\
    \midrule

    \multirow{2}{*}{1.0e-03}
    & C-PINN & 60 & 1.03 & 8.33e-04 & 1.69e-03 \\
    & PINN   & \cellcolor{blue!10}50 & \cellcolor{blue!10}0.84 & 5.92e-04 & 1.37e-03 \\
    \midrule

    \multirow{2}{*}{1.0e-04}
    & C-PINN & \cellcolor{blue!10}110 & \cellcolor{blue!10}1.89 & 7.99e-05 & 1.91e-04 \\
    & PINN   & 450 & 7.65 & 9.87e-05 & 2.37e-04 \\
    \midrule

    \multirow{2}{*}{1.0e-05}
    & C-PINN & \cellcolor{blue!10}530 & \cellcolor{blue!10}9.05 & 9.79e-06 & 3.00e-05 \\
    & PINN   & 4830 & 81.92 & 9.98e-06 & 2.98e-05 \\
    \midrule

    \multirow{2}{*}{1.0e-06}
    & C-PINN & \cellcolor{blue!10}5320 & \cellcolor{blue!10}90.20 & 8.63e-07 & 2.48e-06\\
    & PINN   & 20450 & 346.64 & 9.37e-07 & 2.74e-06 \\
    \bottomrule
    \end{tabular}
\end{table}

\begin{figure}[H]
    \centering
    \includegraphics[width=0.75\linewidth]{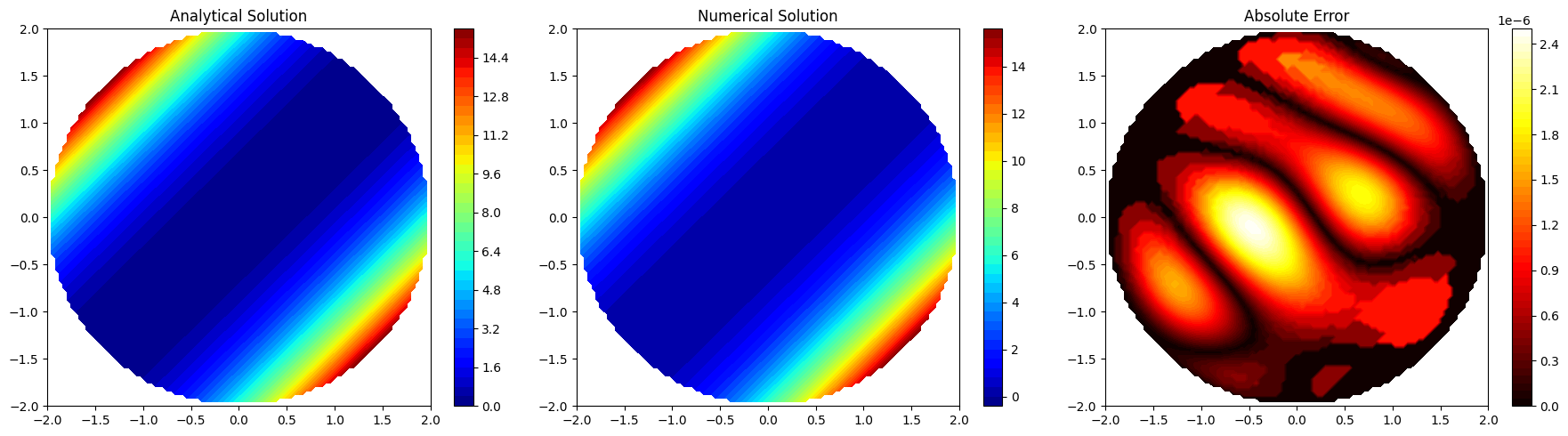}
    \caption{Results for the weakly singular case  of \eqref{eq:ex1}. From left to right: exact solution, numerical solution, and absolute error.}
    \label{fig:ex1_2}
\end{figure}

\begin{figure}[H]
    \centering
    \includegraphics[width=0.75\linewidth]{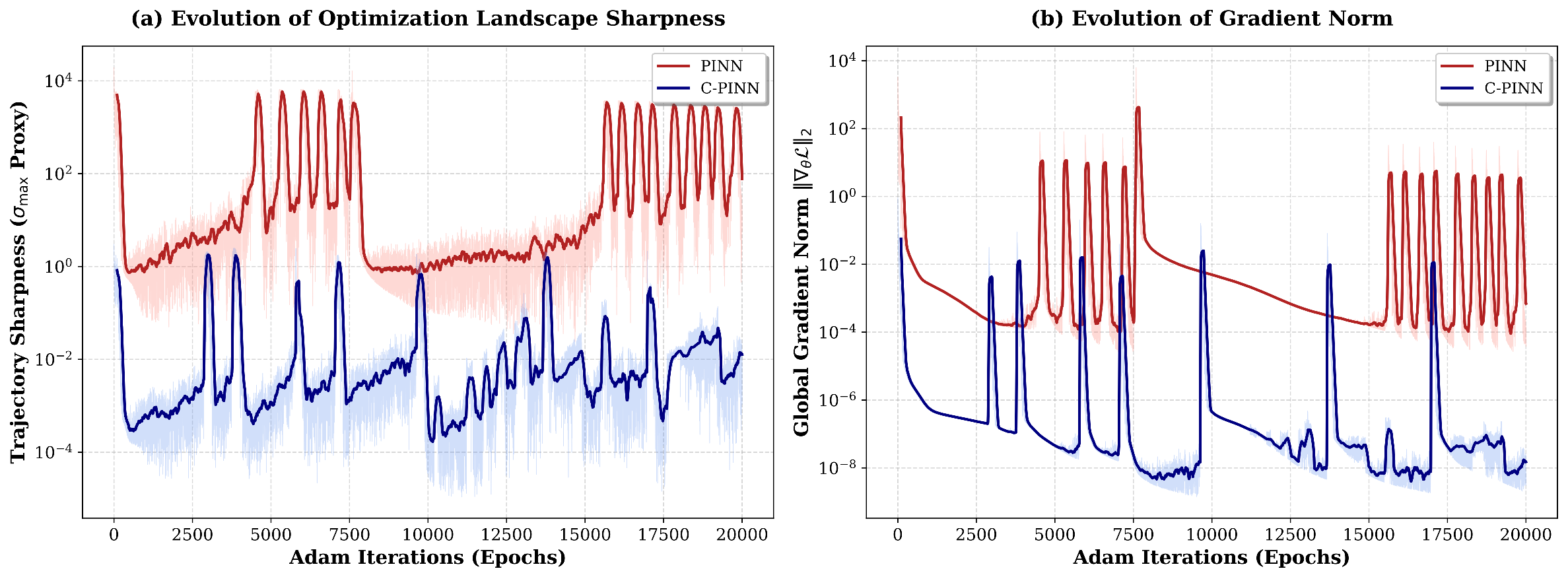}
    \caption{Optimization dynamics comparison: C-PINN vs. PINN for the weakly singular case of \eqref{eq:ex1}.}
    \label{fig:gn2}
\end{figure}

To further demonstrate the robustness and general applicability of the proposed method, we extend the previous example to a more general elliptic equation that includes lower-order terms. Compared to the pure second-order case, this problem involves additional drift and reaction components, making it more challenging. The results show that our method continues to perform well under this more general setting.

\subsection{Example 4.2 General Second-Order Elliptic Equation with Lower-Order Terms}
We consider the following equation which contained additional drift and reaction components:
\begin{equation}
  \begin{cases}
    \mathcal{L} u = A:D^2 u + \bs{b}\cdot\nabla u -c u =f , \ &\text{in} \ \Omega, \\
    u = 0  , \ &\text{on} \ \partial \Omega.
  \end{cases}
\end{equation}

We first consider the case with continuous coefficients. Let
\begin{equation}\label{eq:ex2}
    A=\left( \begin{matrix}
        |\boldsymbol{x}|+1&		-|\boldsymbol{x}|\\
        -|\boldsymbol{x}|&		5|\boldsymbol{x}|+1\\
    \end{matrix} \right),
\end{equation}
the computational domain is set as $\Omega = (-1,1) \times (-1,1)$, and $\bs{b} = (x_1,x_2)^T,c=3$. In this case, the exact solution is set as $u(x_1,x_2)=\sin(2\pi x_1)\sin(\pi x_2)e^{x_1\cos(x_2)}$. In Table \ref{tab:ex2}, we present the numerical cost of the solutions obtained using C-PINN and PINN methods, together with the time required. As is shown in Fig \ref{fig:ex2}, it is evident that our method effectively solves the problem. This accuracy improvement is further corroborated by the optimization dynamics in Fig \ref{fig:gn3}, which demonstrates the significantly enhanced landscape smoothness and gradient stability of C-PINN.

\begin{table}[htbp]
    \centering
    \caption{Computational cost and actual errors for C-PINN and PINN to reach target $l^2$ accuracy milestones for problem \eqref{eq:ex2}.}
    \label{tab:ex2}
    \begin{tabular}{c l | r r c c}
    \toprule
    \multirow{2}{*}{\shortstack{Target \\ $l^2$ Error}} & \multirow{2}{*}{Method} & \multicolumn{2}{c}{Computational Cost} & \multicolumn{2}{c}{Actual Errors} \\
    \cmidrule(lr){3-4} \cmidrule(lr){5-6}

    & & Epochs & Time (s) & \makebox[2.2cm][c]{$l^2$ Error} & \makebox[2.2cm][c]{$l^ \infty$ Error} \\
    \midrule

    \multirow{2}{*}{5.0e-03}
    & C-PINN & \cellcolor{blue!10}1840 & \cellcolor{blue!10}26.39 & 4.99e-03 & 1.70e-02 \\
    & PINN   & 2210 & 33.13 & 4.96e-03 & 1.91e-02 \\
    \midrule

    \multirow{2}{*}{1.0e-03}
    & C-PINN & \cellcolor{blue!10}3910 & \cellcolor{blue!10}56.96 & 9.64e-04 & 3.52e-03 \\
    & PINN   & 4640 & 69.01 & 9.19e-04 & 2.88e-03 \\
    \midrule

    \multirow{2}{*}{5.0e-04}
    & C-PINN & \cellcolor{blue!10}4970 & \cellcolor{blue!10}72.33 & 4.96e-04 & 1.99e-03 \\
    & PINN   & 6030 & 88.84 & 4.90e-04 & 1.61e-03 \\
%
    \bottomrule
    \end{tabular}
\end{table}

\begin{figure}[H]
    \centering
    \includegraphics[width=0.85\linewidth]{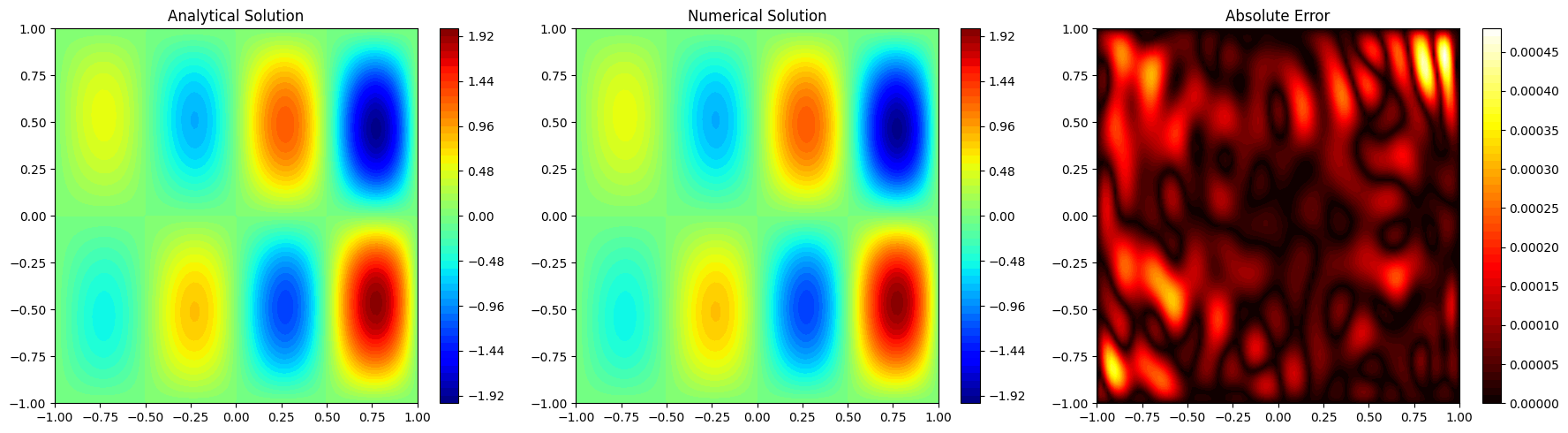}
    \caption{Results for problem Eq. \eqref{eq:ex2}. From left to right: exact solution, numerical solution, and absolute error.}
    \label{fig:ex2}
\end{figure}

\begin{figure}[H]
    \centering
    \includegraphics[width=0.85\linewidth]{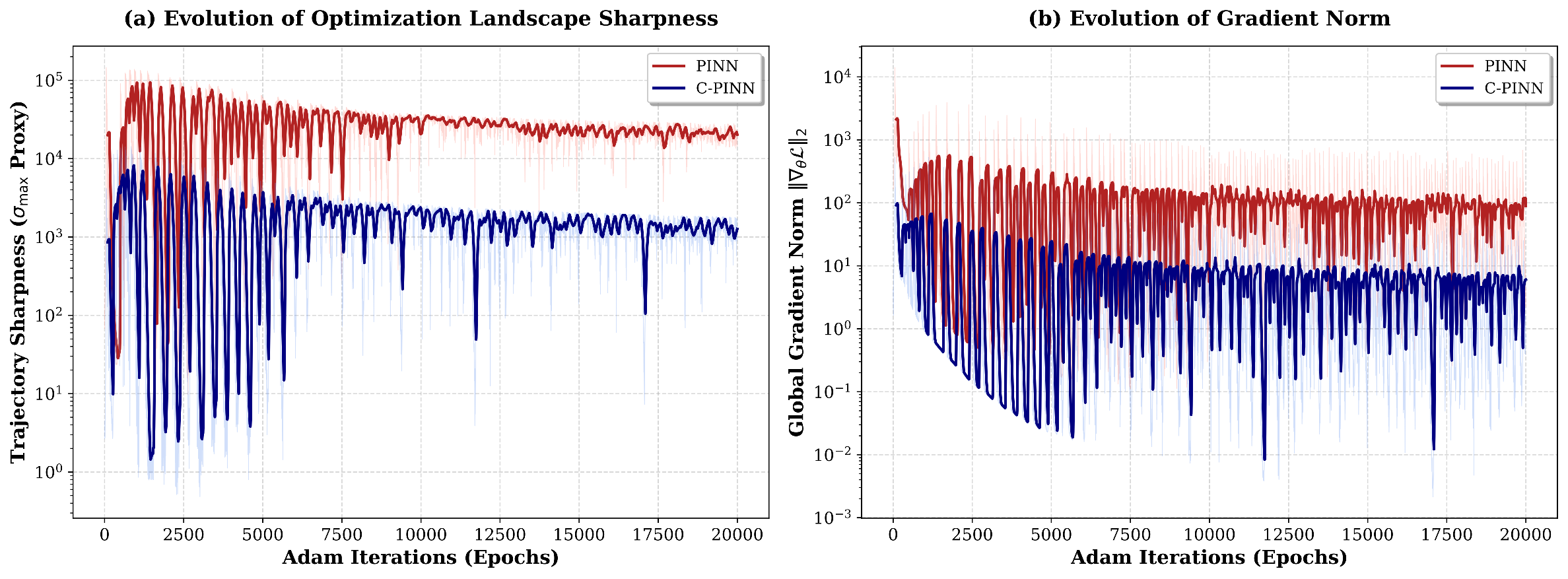}
    \caption{Optimization dynamics comparison: C-PINN vs. PINN for problem Eq. \eqref{eq:ex2}}
    \label{fig:gn3}
\end{figure}
While the case with continuous coefficients provides a baseline, we further investigate the more challenging scenario with discontinuous coefficients to assess the robustness of the proposed method across different regularity regimes. Let
\begin{equation} \label{eq:ex3}
    A=\left( \begin{matrix}
        2&		\frac{x_1x_2}{|x_1x_2|}\\
        \frac{x_1x_2}{|x_1x_2|}&		2\\
    \end{matrix} \right),
\end{equation}
the computational domain is set as $\Omega = (-1,1) \times (-1,1)$, and $\boldsymbol{b} = (x_1,x_2)^T,c=3$. In this case, the exact solution is set as $u(x)=(x_1 e^{1-|x_1|}-x_1)(x_2 e^{1-|x_2|}-x_2)$. Table \ref{tab:ex3} presents the numerical comparison for this case. While the baseline PINN may exhibit a marginal advantage under loose error tolerances, its highly non-convex loss landscape frequently leads to optimization stagnation or aberrant loss spikes. In contrast, our C-PINN fundamentally smooths the loss landscape, thereby demonstrating a pronounced superiority and robust convergence when strict high-accuracy targets are imposed. As is shown in Fig \ref{fig:ex3}, our method also performs well in discontinuous coefficients case. This accuracy improvement is further corroborated by the optimization dynamics in Fig \ref{fig:gn4}, which demonstrates the significantly enhanced landscape smoothness and gradient stability of C-PINN.

\begin{table}[htbp]
    \centering
    \caption{Computational cost and actual errors for C-PINN and PINN to reach target $l^2$ accuracy milestones for problem \eqref{eq:ex3}.}
    \label{tab:ex3}
    \begin{tabular}{c l | r r c c}
    \toprule
    \multirow{2}{*}{\shortstack{Target \\ $l^2$ Error}} & \multirow{2}{*}{Method} & \multicolumn{2}{c}{Computational Cost} & \multicolumn{2}{c}{Actual Errors} \\
    \cmidrule(lr){3-4} \cmidrule(lr){5-6}

    & & Epochs & Time (s) & \makebox[2.2cm][c]{$l^2$ Error} & \makebox[2.2cm][c]{$l^ \infty$ Error} \\
    \midrule

    \multirow{2}{*}{1.0e-03}
    & C-PINN & \cellcolor{blue!10}520 & \cellcolor{blue!10}7.79 & 9.79e-04 & 3.53e-03 \\
    & PINN   & 560 & 8.18 & 8.77e-04 & 2.64e-03 \\
    \midrule

    \multirow{2}{*}{5.0e-04}
    & C-PINN & 900 & 13.43 & 4.55e-04 & 1.75e-03 \\
    & PINN   & \cellcolor{blue!10}780 & \cellcolor{blue!10}12.06 & 4.51e-04 & 2.08e-03 \\
    \midrule

    \multirow{2}{*}{3.0e-04}
    & C-PINN & \cellcolor{blue!10}1220 & \cellcolor{blue!10}18.74 & 2.74e-04 & 9.77e-04 \\
    & PINN   & 8140 & 115.23 & 2.61e-04 & 1.04e-03 \\
    \midrule

    \multirow{2}{*}{1.0e-04}
    & C-PINN & \cellcolor{blue!10}8500 & \cellcolor{blue!10}128.52 & 9.07e-05 & 6.13e-04\\
    & PINN   & 34140 & 537.36 & 9.65e-05 & 5.72e-04 \\
    \bottomrule
    \end{tabular}
\end{table}

\begin{figure}[H]
    \centering
    \includegraphics[width=0.8\linewidth]{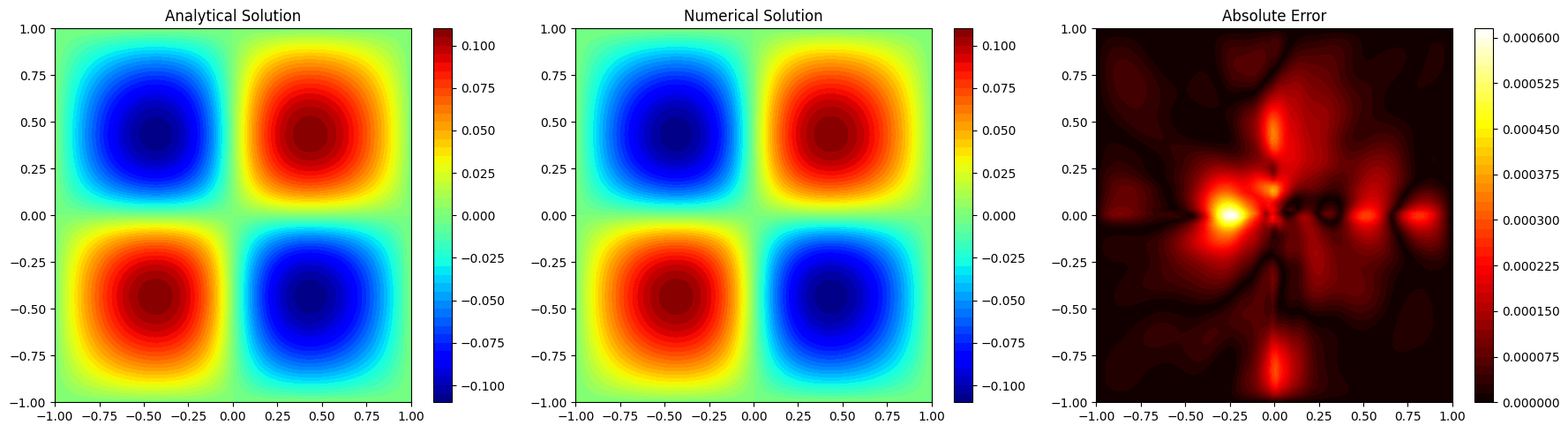}
    \caption{Results for problem Eq. \eqref{eq:ex3}. From left to right: exact solution, numerical solution, and absolute error.}
    \label{fig:ex3}
\end{figure}

\begin{figure}[H]
    \centering
    \includegraphics[width=0.8\linewidth]{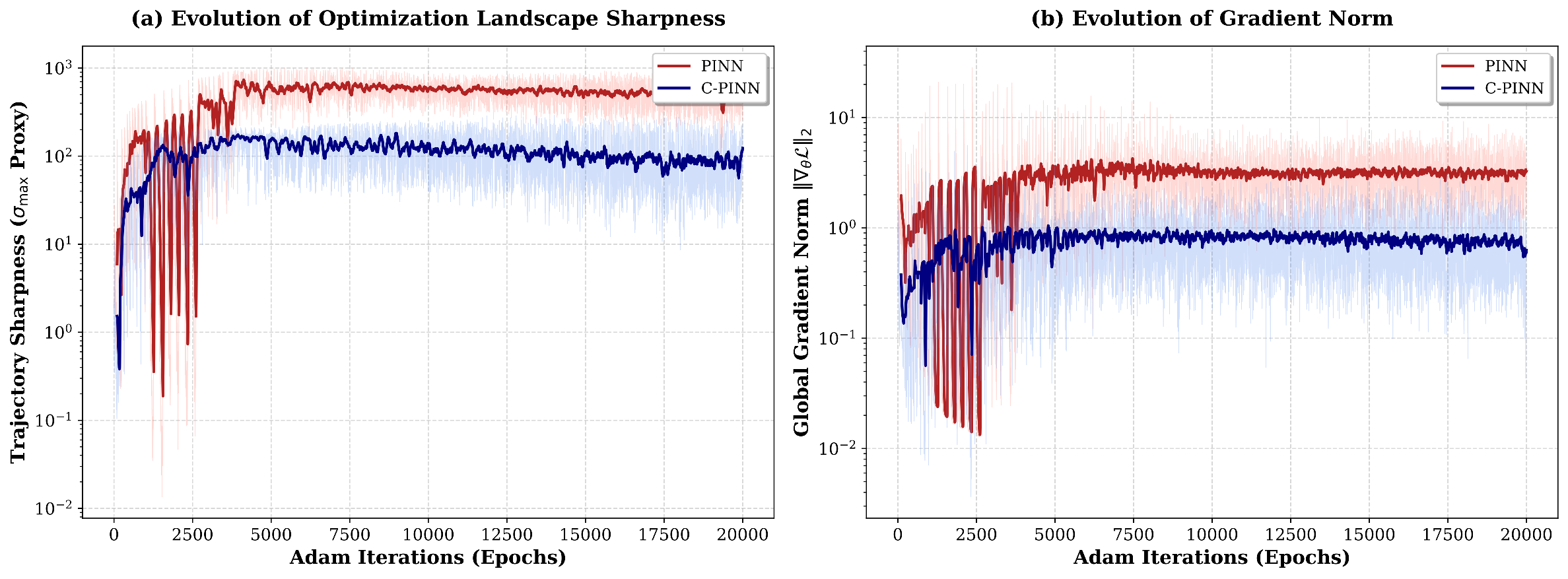}
    \caption{Optimization dynamics comparison: C-PINN vs. PINN for problem Eq. \eqref{eq:ex3}}
    \label{fig:gn4}
\end{figure}

In the preceding examples, the proposed C-PINN framework has demonstrated high accuracy and robust convergence in solving low-dimensional partial differential equations. However, in many advanced practical applications, the governing equations are naturally posed in high-dimensional spaces. Traditional mesh-based numerical methods often become computationally intractable in these scenarios due to the exponential explosion of grid points, famously known as the curse of dimensionality. To further assess the scalability and practical utility of our proposed method, we now transition our focus from low-dimensional configurations to high-dimensional elliptic equations. By leveraging the inherent mesh-free advantage of deep neural networks alongside the mathematically rigorous Cord\`{e}s preconditioning, we aim to corroborate that the C-PINN framework can effectively circumvent the curse of dimensionality and maintain stable, high-fidelity approximations in high-dimensional settings.

\subsection{Example 4.3 High Dimensional Equations}
To evaluate the scalability and robustness of the proposed C-PINN framework in higher dimensions and over complex, non-cuboid geometries, we consider a three-dimensional linear non-divergence elliptic equation. Let the computational domain be an ellipsoid defined by $\Omega = \{(x_1,x_2,x_3) \in \mathbb{R}^3 \mid (x_1/a)^2 + (x_2/b)^2 + (x_3/c)^2 < 1\}$. To ensure the strict ellipticity and the validity of the Cord\`{e}s condition, we constrain the semi-axes such that $\max(a^2, b^2, c^2) = R_{\max}^2 < 9$. In our numerical implementation, the semi-axes of the ellipsoidal domain are specifically chosen as $a = 1.5$, $b = 1.0$, and $c = 0.8$. With this geometric configuration, the maximum squared radius of the domain is bounded by $R_{\max}^2 = 1.5^2 = 2.25$. This bound strictly satisfies the aforementioned theoretical constraint, thereby guaranteeing that the coefficient matrix $A$ maintains the Cord\`{e}s condition uniformly across the entire continuous domain $\Omega$.

The governing equation is given by Eq. \eqref{eq:general_nondiv}, where the strongly heterogeneous coefficient matrix $A(\bs x)$ is constructed as $A(\bs x) = 3I + \bs x\bs x^T$ with $\bs x = (x_1, x_2, x_3)^T$. Explicitly, the matrix takes the following form
\begin{equation} \label{eq:3d_matrix_1}
    A(\bs x) = \begin{pmatrix}
    3 + x_1^2 & x_1 x_2 & x_1 x_3 \\
    x_1 x_2 & 3 + x_2^2 & x_2 x_3 \\
    x_1 x_3 & x_2 x_3 & 3 + x_3^2
    \end{pmatrix}.
\end{equation}

To systematically quantify the approximation error, we construct a synthetic exact solution that smoothly decays to zero at the domain boundary:
\begin{equation} \label{eq:3d_exact_1}
    u(x_1,x_2,x_3) = B(x_1,x_2,x_3) \cdot \exp(x_1x_2x_3),
\end{equation}
where the boundary distance function $B(x_1,x_2,x_3)$ is defined as
\begin{equation}
    B(x_1,x_2,x_3) = 1 - \left(\frac{x_1}{a}\right)^2 - \left(\frac{x_2}{b}\right)^2 - \left(\frac{x_3}{c}\right)^2.
\end{equation}

By this design, the exact solution naturally satisfies the homogeneous Dirichlet boundary condition $u = 0$ on $\partial\Omega$. The corresponding right-hand side source term $f(x_1,x_2,x_3)$ is derived analytically by substituting Eq. \eqref{eq:3d_exact_1} into the differential operator. This benchmark problem presents a comprehensive challenge for neural network solvers due to its elevated dimensionality, highly varying coefficient matrix, and non-trivial curved boundary.

To better illustrate the internal characteristics of the 3D solution, Fig. \ref{fig:3d_1} presents a cutaway rendering of the predicted field, where one quadrant has been removed to expose the center. In Table \ref{tab:3d_1}, we present the numerical errors of the solutions of this case. This accuracy improvement is further corroborated by the optimization dynamics in Fig \ref{fig:gn5}, which demonstrates the significantly enhanced landscape smoothness and gradient stability of C-PINN.

\begin{table}[htbp]
    \centering
    \caption{Computational cost and actual errors for C-PINN and PINN to reach target $l^2$ accuracy milestones for problem \eqref{eq:3d_exact_1}.}
    \label{tab:3d_1}
    \begin{tabular}{c l | r r c c}
    \toprule
    \multirow{2}{*}{\shortstack{Target \\ $l^2$ Error}} & \multirow{2}{*}{Method} & \multicolumn{2}{c}{Computational Cost} & \multicolumn{2}{c}{Actual Errors} \\
    \cmidrule(lr){3-4} \cmidrule(lr){5-6}

    & & Epochs & Time (s) & \makebox[2.2cm][c]{$l^2$ Error} & \makebox[2.2cm][c]{$l^ \infty$ Error} \\
    \midrule

    \multirow{2}{*}{5.0e-03}
    & C-PINN & \cellcolor{blue!10}810 & \cellcolor{blue!10}27.07 & 4.50e-03 & 1.41e-02 \\
    & PINN   & 850 & 28.31 & 4.62e-03 & 1.41e-02 \\
    \midrule

    \multirow{2}{*}{1.0e-03}
    & C-PINN & \cellcolor{blue!10}2160 & \cellcolor{blue!10}71.78 & 9.98e-04 & 3.57e-03 \\
    & PINN   & 2600 & 86.42 & 9.98e-04 & 2.96e-03 \\
    \midrule

    \multirow{2}{*}{5.0e-04}
    & C-PINN & \cellcolor{blue!10}3510 & \cellcolor{blue!10}116.71 & 4.98e-04 & 1.42e-03 \\
    & PINN   & 4250 & 141.29 & 4.93e-04 & 1.39e-03 \\
    \midrule

    \multirow{2}{*}{1.0e-04}
    & C-PINN & \cellcolor{blue!10}6010 & \cellcolor{blue!10}199.69 & 9.97e-05 & 3.55e-04\\
    & PINN   & 10850 & 359.62 & 9.99e-05 & 4.13e-04 \\
    \bottomrule
    \end{tabular}
\end{table}

\begin{figure}[H]
    \centering
    \includegraphics[width=1.0\linewidth]{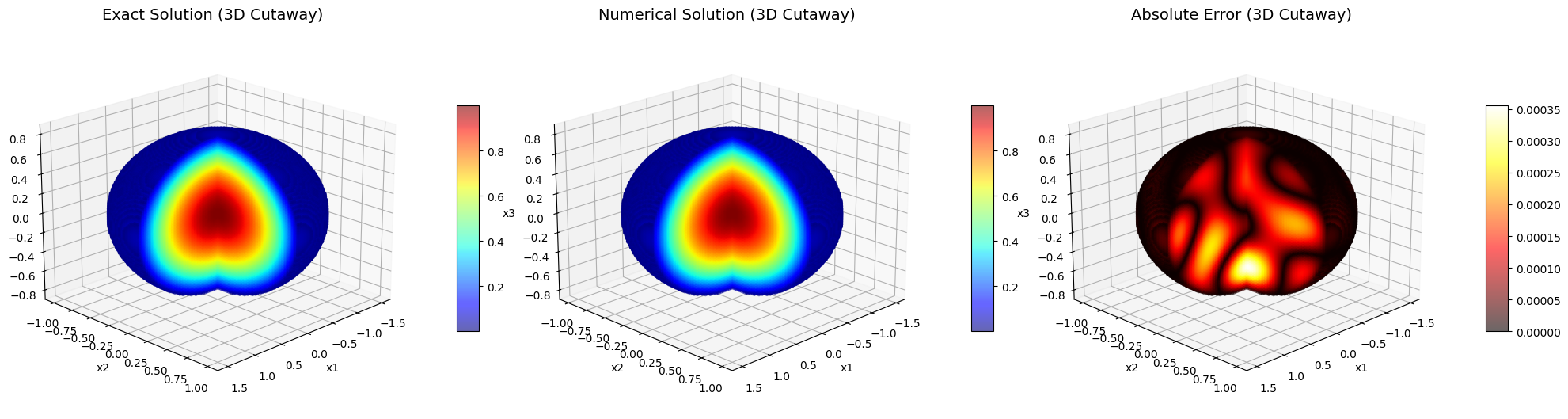}
    \caption{Results for problem Eq. \eqref{eq:3d_exact_1}. From left to right: exact solution, numerical solution, and absolute error.}
    \label{fig:3d_1}
\end{figure}
\vspace{-0.2in}
\begin{figure}[H]
    \centering
    \includegraphics[width=0.8\linewidth]{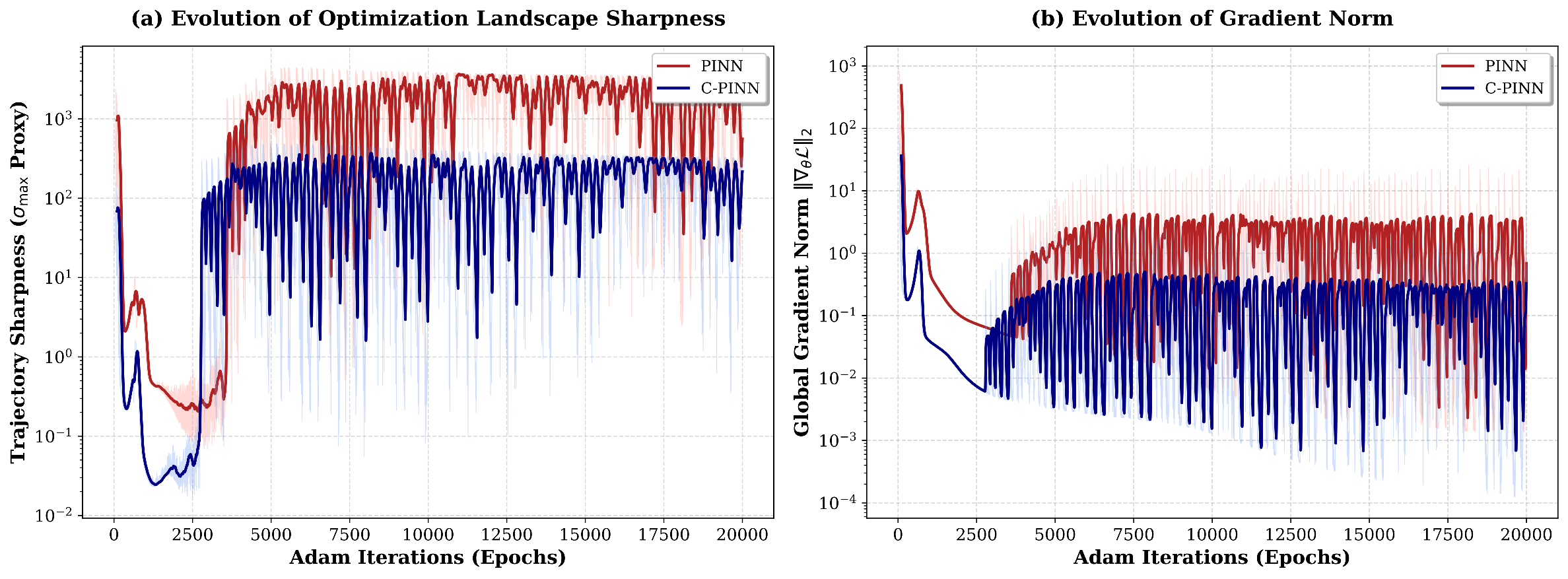}
    \caption{Optimization dynamics comparison: C-PINN vs. PINN for problem Eq. \eqref{eq:3d_exact_1}}
    \label{fig:gn5}
\end{figure}

While the preceding 3D experiments demonstrated C-PINN's capability to handle complex geometries and low-regularity features, a critical advantage of neural-network-based solvers is their potential to overcome the curse of dimensionality. To evaluate this scalability, our subsequent experiments consider higher-dimensional equations, where traditional grid-based methods become computationally intractable.

Consider the following 5D case:
\begin{equation} \label{eq:5d}
    A(\bs x) = (a_{ij})_{5 \times 5} = \begin{cases} 5, & \text{if } i = j \\ \frac{x_i x_j}{|x_ix_j|}, & \text{if } i \neq j \end{cases} \quad \text{for } i, j \in \{1, 2, \dots, 5\},
\end{equation}
which satisfies Eq. \eqref{eq:general_nondiv}, and the exact solution is $u(\bs x) = \prod_{i=1}^{5} \cos\left(\frac{\pi}{2} x_i\right) + \prod_{i=1}^{5} \sin(\pi x_i)$ with domain $\Omega = (-1,1)^5$. Obviously, the exact solution can satisfy $u=0$ on $\partial \Omega$.

Similar to the 3D discontinuous scenario, it is essential to rigorously verify the Cord\`{e}s condition for this 5D high-dimensional operator. For almost every $\bs x \in \Omega$, the squared values of the off-diagonal elements are precisely $1$. Given the spatial dimension $d=5$, we can compute the trace and the squared Frobenius norm of $A(\bs x)$ as constants:
\begin{align*}
    \text{tr}(A) &= \sum_{i=1}^5 5 = 25, \\
    \|A\|_F^2 &= \sum_{i=1}^5 5^2 + \sum_{i \neq j} \left( \frac{x_i x_j}{|x_i x_j|} \right)^2 = 5 \times 25 + (5^2 - 5) \times 1 = 125 + 20 = 145.
\end{align*}
According to the generalized Cord\`{e}s condition for dimension $d=5$, the required lower bound for strict ellipticity is $d - 1 = 4$. Evaluating the Cord\`{e}s ratio yields
\begin{equation}
    \frac{(\text{tr}(A))^2}{\|A\|_F^2} = \frac{25^2}{145} = \frac{625}{145} = \frac{125}{29} \approx 4.310.
\end{equation}
Since $\frac{125}{29} > 4$, there exists a strictly positive constant $\varepsilon = \frac{9}{29}$ such that ${(\text{tr}(A))^2}/{\|A\|_F^2} = 4 + \varepsilon$. This unequivocally demonstrates that the highly discontinuous 5D matrix $A(\bs x)$ uniformly satisfies the Cord\`{e}s condition almost everywhere in the hypercube domain $\Omega$, thereby theoretically guaranteeing the stable convergence of the C-PINN methodology even in high-dimensional and low-regularity settings. In Table \ref{tab:5d}, we present the numerical errors of the solutions of this case. This accuracy improvement is further corroborated by the optimization dynamics in Fig \ref{fig:gn6}, which demonstrates the significantly enhanced landscape smoothness and gradient stability of C-PINN.

\begin{table}[htbp]
    \centering
    \caption{Computational cost and actual errors for C-PINN and PINN to reach target $l^2$ accuracy milestones for problem \eqref{eq:5d}.}
    \label{tab:5d}
    \begin{tabular}{c l | r r c c}
    \toprule
    \multirow{2}{*}{\shortstack{Target \\ $l^2$ Error}} & \multirow{2}{*}{Method} & \multicolumn{2}{c}{Computational Cost} & \multicolumn{2}{c}{Actual Errors} \\
    \cmidrule(lr){3-4} \cmidrule(lr){5-6}

    & & Epochs & Time (s) & \makebox[2.2cm][c]{$l^2$ Error} & \makebox[2.2cm][c]{$l^ \infty$ Error} \\
    \midrule

    \multirow{2}{*}{5.0e-02}
    & C-PINN & \cellcolor{blue!10}400 & \cellcolor{blue!10}31.71 & 3.62e-02 & 2.58e-01 \\
    & PINN   & 720 & 56.96 & 4.30e-02 & 5.72e-01 \\
    \midrule

    \multirow{2}{*}{1.0e-02}
    & C-PINN & \cellcolor{blue!10}560 & \cellcolor{blue!10}44.37 & 8.95e-03 & 6.79e-02 \\
    & PINN   & 1020 & 80.67 & 9.97e-03 & 1.29e-01 \\
    \midrule

    \multirow{2}{*}{5.0e-03}
    & C-PINN & \cellcolor{blue!10}840 & \cellcolor{blue!10}66.51 & 4.95e-03 & 5.02e-02 \\
    & PINN   & 2040 & 141.29 & 4.95e-03 & 4.31e-02 \\
    \midrule

    \multirow{2}{*}{1.0e-03}
    & C-PINN & \cellcolor{blue!10}7980 & \cellcolor{blue!10}631.44 & 9.84e-04 & 1.26e-02\\
    & PINN   & 18260 & 1444.14 & 9.99e-04 & 8.60e-03 \\
    \bottomrule
    \end{tabular}
\end{table}

In the previous examples, the exact solution is known, which allows for a direct evaluation of the approximation error. However, in many practical problems, the exact solution is unavailable. To better assess the performance of the proposed method in such realistic scenarios, we consider elliptic equations with a constant source term and unknown exact solutions. In addition, both continuous and discontinuous coefficient cases are included to evaluate the robustness of the method under different regularity settings.

\begin{figure}[H]
    \centering
    \includegraphics[width=0.8\linewidth]{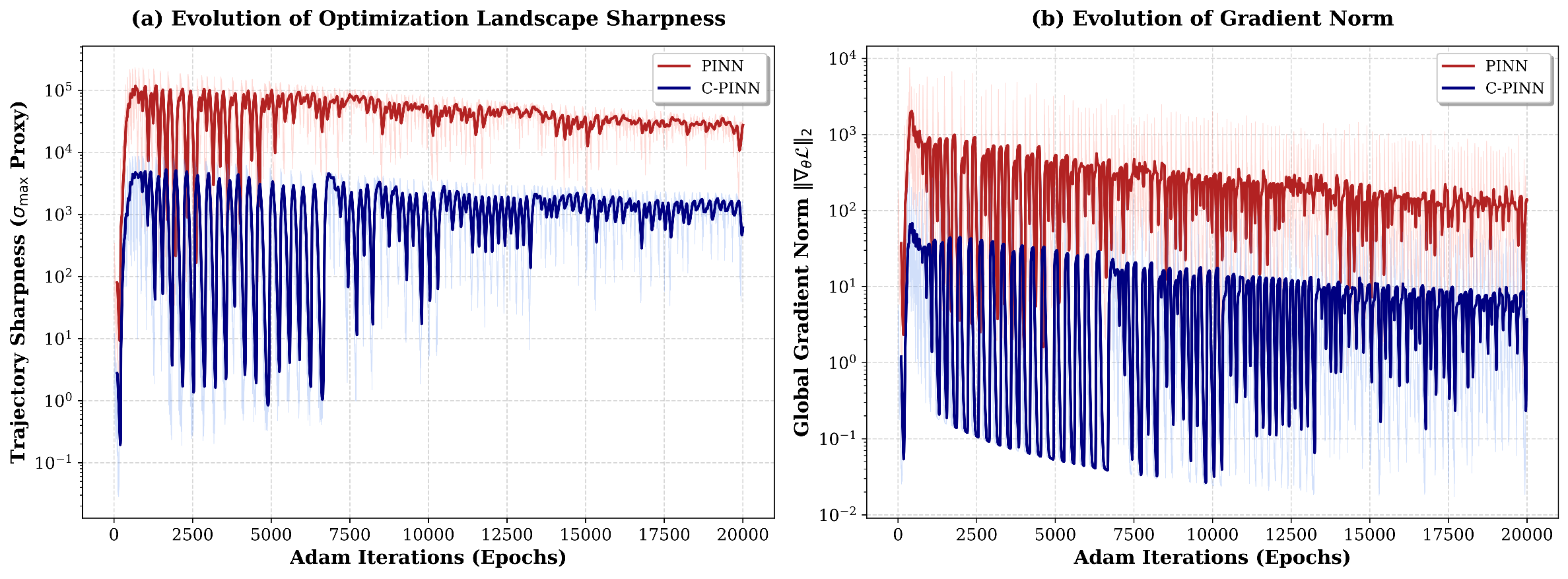}
    \caption{Optimization dynamics comparison: C-PINN vs. PINN for problem Eq. \eqref{eq:5d}}
    \label{fig:gn6}
\end{figure}

\subsection{Example 4.4 No Knowing the Exact Solution Case}
Consider the following equation:
\begin{equation}
  \begin{cases}
    \mathcal{L} u = A:D^2 u + \bs{b}\cdot \nabla u -cu =f  , \ &\text{in} \ \Omega, \\
    u = 0  , \ &\text{on} \ \partial \Omega.
  \end{cases}
\end{equation}
First, we consider the continuous coefficients with unknown solution case. Let $\Omega = (-\frac{1}{2},\frac{1}{2}) \times (-\frac{1}{2},\frac{1}{2})$, and the coefficients are specified as $\boldsymbol{b} = (x_1, x_2)^T,c=4$, and
\begin{equation} \label{eq:ex4}
    A=\left( \begin{matrix}
    |\boldsymbol{x}|+2&		-|\boldsymbol{x}|\\
    -|\boldsymbol{x}|&		3|\boldsymbol{x}|+2\\
    \end{matrix} \right) .
\end{equation}
The source term is set to $f=2$. To evaluate the performance of the proposed approach, Fig \ref{fig:ex4} illustrates the numerical solutions generated by C-PINN and Chebyshev spectral method. The clear agreement between the two plots confirms the effectiveness of our method in accurately capturing the solution behavior for this class of equations.

\begin{figure}[H]
    \centering
    \includegraphics[width=0.75\linewidth]{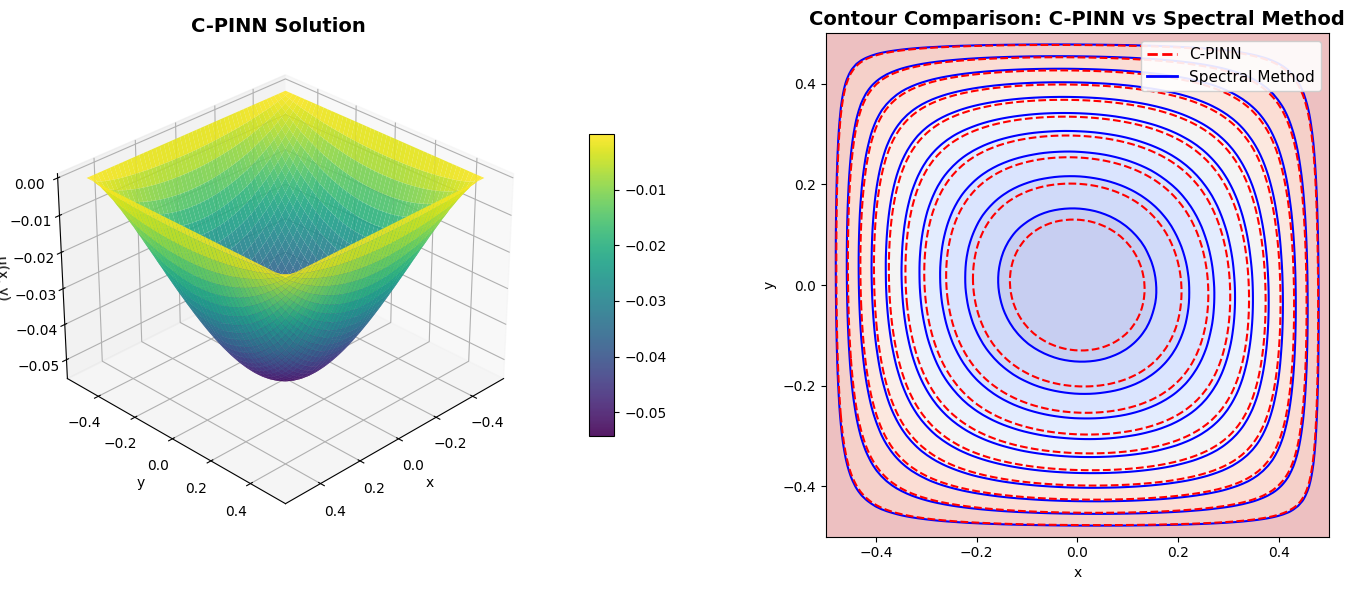}
    \caption{Comparison of numerical solutions for problem \eqref{eq:ex4}. Left: Results obtained via C-PINN; Right: Results comparison between C-PINN and Chebyshev spectral method}
    \label{fig:ex4}
\end{figure}

After that, we consider the discontinuous coefficients with unknown solution. In this case, we consider the domain to be $\Omega = (-1,1) \times (-1,1)$, and the coefficient is $\boldsymbol{b} = (x_1,x_2)^T, c = 3$, and
\begin{equation}\label{eq:ex5}
    A=\left( \begin{matrix}
    2&		\frac{x_1x_2}{|x_1x_2|}\\
    \frac{x_1x_2}{|x_1x_2|}&		2\\
    \end{matrix} \right) .
\end{equation}
The source term is set to $f=2$. The Fig \ref{fig:ex5} shows the numerical solution of problem \eqref{eq:ex5} by C-PINN method and Chebyshev spectral method. The above results confirm that the proposed method remains effective in more realistic settings where the exact solution is unavailable. Moreover, its stable performance across both continuous and discontinuous coefficient cases highlights its robustness with respect to varying regularity, indicating strong potential for practical applications.

\begin{figure}[H]
    \centering
    \includegraphics[width=0.75\linewidth]{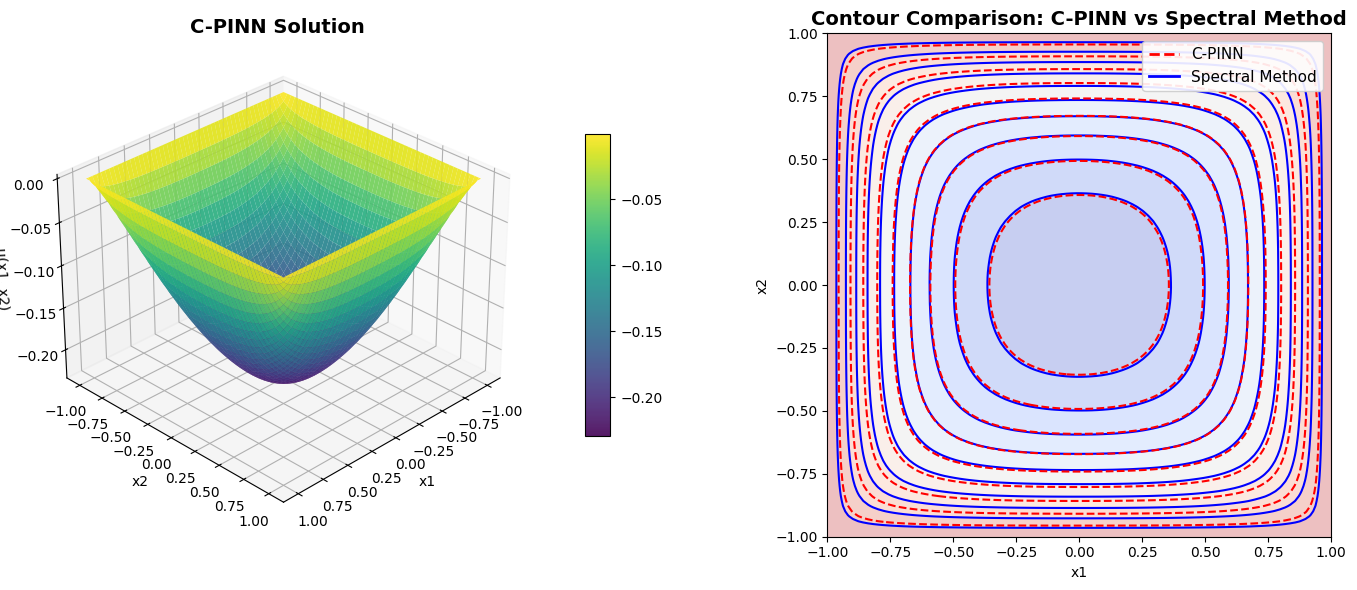}
    \caption{Comparison of numerical solutions for problem \eqref{eq:ex5}. Left: Results obtained via C-PINN; Right: Results comparison between C-PINN and Chebyshev spectral method}
    \label{fig:ex5}
\end{figure}

Building upon the previous experiments on linear elliptic equations, we further investigate the performance of the proposed method on HJB equations. As a prototypical class of fully nonlinear PDEs arising in optimal control, HJB equations pose significant challenges for numerical methods. This setting provides a more stringent test for evaluating the robustness and effectiveness of our approach.

\subsection{Example 4.5 Hamilton-Jacobi-Bellman Equations}
In this experiment, we solve the following nonlinear HJB equations on the domain $\Omega = (-\pi, \pi) \times (-\pi, \pi)$:
\begin{equation}\label{eq:ex6}
  \begin{cases}
    \underset{\alpha \in \varLambda}{\mathrm{sup}}\left( \mathcal{L} ^{\alpha}u-f^{\alpha} \right) =0 , \ &\mathrm{in} \ \Omega ,\\
    u=0 , \ &\mathrm{on} \ \partial \Omega,\\
  \end{cases}
\end{equation}
where
\begin{equation}
    \mathcal{L} ^{\alpha}u=A^{\alpha}(x):D^2u+\boldsymbol{b}^{\alpha}(x)\cdot \nabla u-c^{\alpha}(x)u,
\end{equation}
and
\begin{equation}
    A^{\alpha}=\begin{cases}
        \left( \begin{matrix}
        2 &		\ 1/2\\
        1/2&	\	3/2\\
    \end{matrix} \right) +\frac{x_1x_2}{|x_1||x_2|}\left( \begin{matrix}
        1&	\	1/2\\
        1/2&	\	1/2\\
    \end{matrix} \right) ,&		\alpha =1,\\
        \left( \begin{matrix}
        3/2&	\	1/2\\
        1/2&	\	2\\
    \end{matrix} \right) +\frac{x_1x_2}{|x_1||x_2|}\left( \begin{matrix}
        1/2&	\	1/2\\
        1/2&	\	1\\
    \end{matrix} \right) ,&		\alpha =2.\\
    \end{cases}
\end{equation}
We take the index set $ \varLambda =\{ 1,2 \}$ and define the coefficients as $\boldsymbol{b}_1 = \boldsymbol{b}_2 = (1,0), c_1 = c_2 = 1 $. The source functions $f_1$ and $f_2$ are chosen to fit the exact solution $ u(x_1 , x_2) = \sin(x_1) \sin(x_2) $. In Table \ref{tab:ex6}, we present the numerical cost of the solutions obtained using C-PINN and PINN methods. As is shown in Fig \ref{fig:ex6}, it is evident that our method effectively solves HJB equations. This accuracy improvement is further corroborated by the optimization dynamics in Fig. \ref{fig:gn8}, which demonstrates the significantly enhanced landscape smoothness and gradient stability of C-PINN.

\begin{table}[htbp]
    \centering
    \caption{Computational cost and actual errors for C-PINN and PINN to reach target $l^2$ accuracy milestones for problem \eqref{eq:ex6}.}
    \label{tab:ex6}
    \begin{tabular}{c l | r r r c c}
    \toprule
    \multirow{2}{*}{\shortstack{Target \\ $l^2$ Error}} & \multirow{2}{*}{Method} & \multicolumn{3}{c}{Computational Cost} & \multicolumn{2}{c}{Actual Errors} \\
    \cmidrule(lr){3-5} \cmidrule(lr){6-7}

    & & Epochs & Newton Iters & Time (s) & \makebox[2.2cm][c]{$l^2$ Error} & \makebox[2.2cm][c]{$l^ \infty$ Error} \\
    \midrule

    \multirow{2}{*}{1.0e-03}
    & C-PINN & \cellcolor{blue!10}1010 & 2 & \cellcolor{blue!10}14.74 & 8.87e-04 & 3.65e-03 \\
    & PINN   & 1420 & - & 22.35 & 9.91e-04 & 4.32e-03 \\
    \midrule

    \multirow{2}{*}{5.0e-04}
    & C-PINN & \cellcolor{blue!10}1390 & 2 & \cellcolor{blue!10}20.16 & 4.85e-04 & 2.00e-03 \\
    & PINN   & 1810 & - & 28.74 & 4.96e-04 & 2.07e-03 \\
    \midrule

    \multirow{2}{*}{1.0e-04}
    & C-PINN & \cellcolor{blue!10}13300 & 14 & \cellcolor{blue!10}190.08 & 9.74e-05 & 3.48e-04 \\
    & PINN   & 17080 & - & 268.84 & 9.60e-05 & 3.16e-04 \\
    \midrule

    \multirow{2}{*}{5.0e-05}
    & C-PINN & \cellcolor{blue!10}23170 & 24 & \cellcolor{blue!10}331.04 & 4.21e-05 & 1.31e-04\\
    & PINN   & 36800 & - & 576.98 & 4.85e-05 & 1.67e-04 \\
    \bottomrule
    \end{tabular}
\end{table}

\begin{figure}[H]
    \centering
    \includegraphics[width=0.8\linewidth]{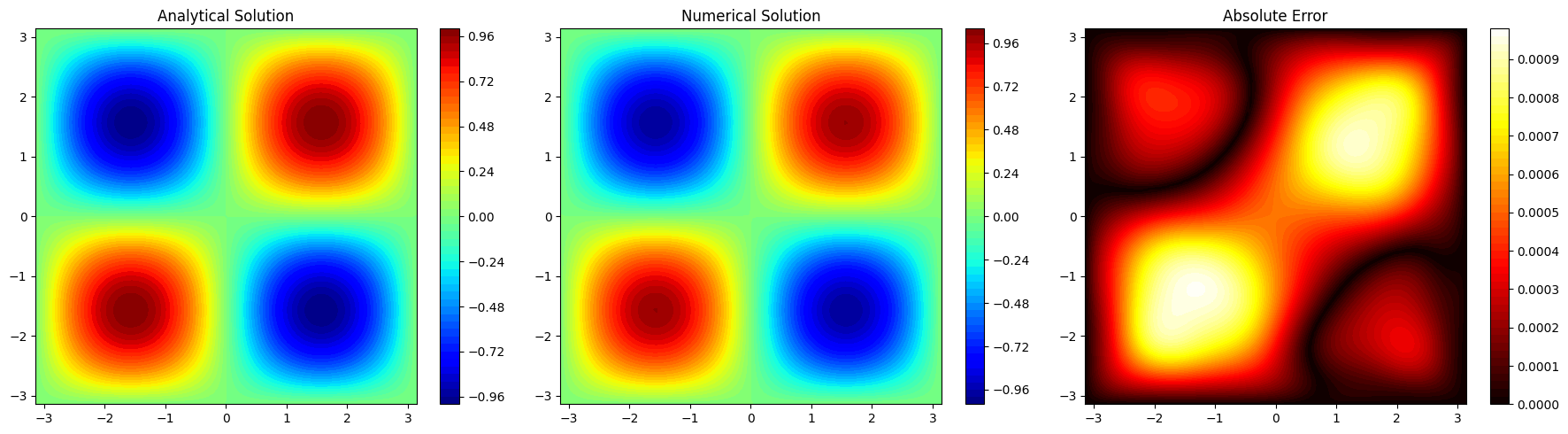}
    \caption{Results for problem Eq. \eqref{eq:ex6}. From left to right: exact solution, numerical solution, and absolute error.}
    \label{fig:ex6}
\end{figure}

\begin{figure}[H]
    \centering
    \includegraphics[width=0.7\linewidth]{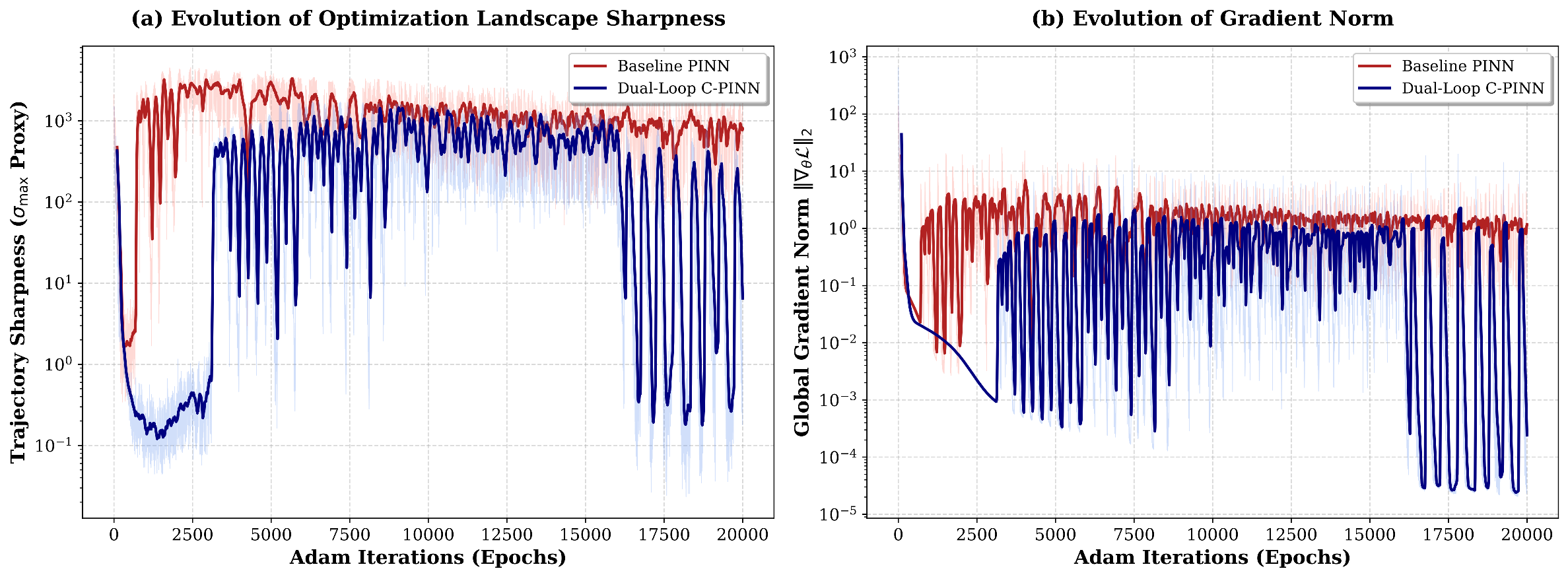}
    \caption{Optimization dynamics comparison: C-PINN vs. PINN for problem Eq. \eqref{eq:ex6}}
    \label{fig:gn8}
\end{figure}

Building upon the results for HJB equations, we further investigate the MA equations, which serves as a prototypical example of fully nonlinear PDEs with strong structural constraints. Compared to HJB equations, the MA equations introduce additional challenges due to its highly nonlinear and degenerate nature. This example provides a more stringent benchmark for evaluating the robustness and generality of the proposed method.

\subsection{Example 4.6 Monge-Amp\`{e}re Equations}
We consider the smooth radial function:
\begin{equation}
    u(x_1,x_2)=e^{\frac{x_1^2+x_2^2}{2}},
\end{equation}
on the unit square $\Omega = (0,1) \times (0,1)$. A straight forward calculation demonstrates that
\begin{equation}\label{eq:ma1}
    f(x_1,x_2):= \det (D^2 u(x_1,x_2))= (1+x_1^2+x_2^2)e^{x_1^2+x_2^2}.
\end{equation}
Denote $g$ as the restriction of $u$ to $\partial \Omega$, where $u$ is a convex solution to the Dirichlet problem of \eqref{eq1} with right-hand side $f$ defined by \eqref{eq:ma1} and boundary values $g$. In Table \ref{tab:ma}, we present the numerical cost of the solutions obtained using C-PINN and PINN methods. As is shown in Fig \ref{fig:ex7}, it is evident that our method effectively solves MA equations with Dirichlet boundary condition. This accuracy improvement is further corroborated by the optimization dynamics in Fig. \ref{fig:gn9}, which demonstrates the significantly enhanced landscape smoothness and gradient stability of C-PINN.

\begin{table}[htbp]
    \centering
    \caption{Computational cost and actual errors for C-PINN and PINN to reach target $l^2$ accuracy milestones for problem \eqref{eq:ma1}.}
    \label{tab:ma}
    \begin{tabular}{c l | r r r c c}
    \toprule
    \multirow{2}{*}{\shortstack{Target \\ $l^2$ Error}} & \multirow{2}{*}{Method} & \multicolumn{3}{c}{Computational Cost} & \multicolumn{2}{c}{Actual Errors} \\
    \cmidrule(lr){3-5} \cmidrule(lr){6-7}

    & & Epochs & Newton Iters & Time (s) & \makebox[2.2cm][c]{$l^2$ Error} & \makebox[2.2cm][c]{$l^ \infty$ Error} \\
    \midrule

    \multirow{2}{*}{1.0e-03}
    & C-PINN & 10490 & 10 & 110.52 & 9.90e-04 & 2.33e-03 \\
    & PINN   & \cellcolor{blue!10}7720 & - & \cellcolor{blue!10}80.67 & 9.64e-04 & 2.49e-03 \\
    \midrule

    \multirow{2}{*}{5.0e-04}
    & C-PINN & \cellcolor{blue!10}11500 & 11 & \cellcolor{blue!10}121.07 & 4.90e-04 & 1.86e-03 \\
    & PINN   & 22170 & - & 232.62 & 4.99e-04 & 1.39e-03 \\
    \midrule

    \multirow{2}{*}{1.0e-04}
    & C-PINN & \cellcolor{blue!10}24630 & 24 & \cellcolor{blue!10}259.68 & 9.75e-05 & 4.12e-04 \\
    & PINN   & \cellcolor{gray!10}- & - & \cellcolor{gray!10}- & (4.80e-04) & (8.64e-04) \\
    \midrule

    \multirow{2}{*}{5.0e-05}
    & C-PINN & \cellcolor{blue!10}28590 & 28 & \cellcolor{blue!10}301.35 & 4.92e-05 & 3.70e-04\\
    & PINN   & \cellcolor{gray!10}- & - & \cellcolor{gray!10}- & (4.80e-04) & (8.64e-04)  \\
    \bottomrule
    \end{tabular}
\end{table}
\vspace{-0.2in}
\begin{figure}[H]
    \centering
    \includegraphics[width=0.85\linewidth]{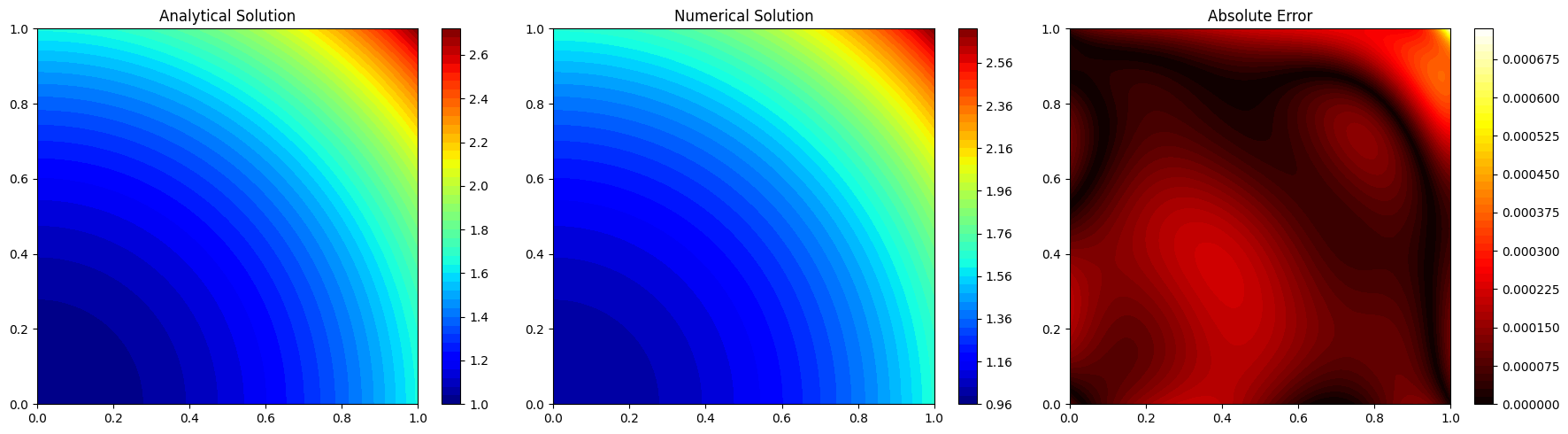}
    \caption{Results for problem Eq. \eqref{eq:ma1}. From left to right: exact solution, numerical solution, and absolute error.}
    \label{fig:ex7}
\end{figure}

\begin{figure}[H]
    \centering
    \includegraphics[width=0.8\linewidth]{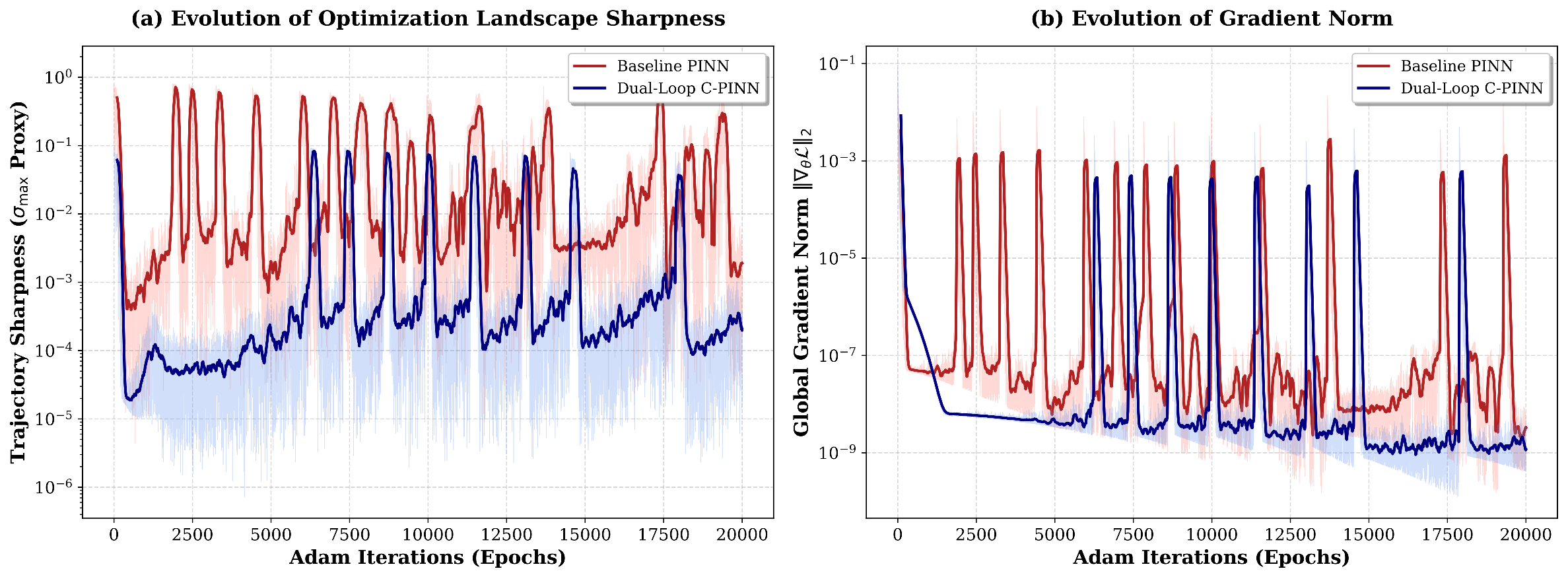}
    \caption{Optimization dynamics comparison: C-PINN vs. PINN for problem Eq. \eqref{eq:ma1}}
    \label{fig:gn9}
\end{figure}

Building upon the numerical results for the Monge-Amp\`{e}re equation with Dirichlet boundary conditions, we further explore its role in optimal transport problems. As is well known, the Monge-Amp\`{e}re equation provides a fundamental link between convex potential functions and optimal transport maps. By considering this application, we aim to demonstrate not only the accuracy but also the practical relevance of the proposed method.

\section{Applications of the Optimal Transport Problem}

The Monge-Amp\`{e}re equation is closely connected to optimal transport theory. In particular, it arises naturally in the characterization of optimal transport maps between probability measures. Let $\mu$ and $\nu$ be two probability measures defined on domains $\Omega$ and $\Omega'$ respectively. The goal of optimal transport is to find a mapping $T:\Omega \rightarrow \Omega'$ that pushes $\mu$ forward to $\nu$ while minimizing a given cost function. A fundamental result in this theory is Brenier's theorem \cite{Brenier1991}, which establishes the existence and structure of the optimal transport map for the quadratic cost.

\begin{theorem} [Brenier]
    Let $\mu$ be absolutely continuous with respect to the Lebesgue measure. Then there exists a convex function $\phi$ such that the optimal transport map $T$ from $\mu$ to $\nu$ is given by
    \begin{equation}
        T(\bs x) = \nabla \phi(\bs x).
    \end{equation}
    Moreover, $\phi$ satisfies the Monge-Amp\`{e}re equation
    \begin{equation}
        \det (D^2 \phi(\bs x)) = \frac{\mu (\bs x)}{\nu (\nabla \phi(\bs x)) }.
    \end{equation}
    subject to the second boundary value condition
    \begin{equation}
        \nabla \phi(V) = V.
    \end{equation}
\end{theorem}

This result establishes a direct link between optimal transport and the Monge-Amp\`{e}re equation, where the transport map can be recovered from the gradient of a convex potential. Therefore, solving the Monge-Amp\`{e}re equation provides a principled approach to computing optimal transport maps.

To efficiently compute the transport map, we apply the Newton-Kantorovich linearization to the highly nonlinear Monge-Amp\`{e}re equation. By rewriting the equation as
\begin{equation}
    F(\phi) = \det(D^2 \phi) \nu(\nabla \phi) - \mu = 0,
\end{equation}
we seek an iterative update $\phi = \phi_k + \delta \phi$, where $\phi_k$ is the better-conditioned potential at the $k$-th iteration. Applying the Fr\'{e}chet derivative yields the linearized equation for the increment $\delta \phi$
\begin{equation}
    \left[ \nu(\nabla \phi_k) \text{cof}(D^2 \phi_k) \right] : D^2 \delta \phi + \left[ \det(D^2 \phi_k) \nabla_{\bs y} \nu(\nabla \phi_k) \right] \cdot \nabla \delta \phi = \mu(\bs x) - \det(D^2 \phi_k) \nu(\nabla \phi_k),
\end{equation}
where $\text{cof}(D^2 \phi_k)$ is the cofactor matrix of the Hessian, and $\nabla_{\bs y} \nu$ denotes the gradient of the target density with respect to the mapped coordinates $\bs y = \nabla \phi_k(\bs x)$. Remarkably, this linearization elegantly transforms the non-local density term $\nu(\nabla \phi)$ into a first-order convection term. Consequently, the second-order principal part depends solely on the structure of $\phi_k$. This allows the Cord\`{e}s preconditioner to be seamlessly applied to the principal coefficient matrix $\nu(\nabla \phi_k) \text{cof}(D^2 \phi_k)$ without being compromised by the first-order convection.

To guarantee the theoretical applicability of the Cord\`{e}s condition, the principal coefficient matrix $A^{(k)}(\bs x) = \nu(\nabla \phi_k) \text{cof}(D^2 \phi_k)$ must be uniformly elliptic. According to Brenier's theorem, a valid optimal transport potential $\phi_k$ is inherently better-conditioned, which ensures that its Hessian matrix $D^2 \phi_k$ is symmetric positive definite (SPD). Consequently, its cofactor matrix $\text{cof}(D^2 \phi_k)$ is also strictly SPD. Given that the target density $\nu$ is bounded and strictly positive, the coefficient matrix $A^{(k)}(\bs x)$ remains SPD throughout the domain $\Omega$. In any spatial dimension $n \ge 2$, an SPD matrix intrinsically satisfies the unconditional Cord`{e}s property, satisfying the inequality
\begin{equation}
    \frac{(\text{tr}(A^{(k)}))^2}{\text{tr}((A^{(k)})^2)} \ge n - 1 + \varepsilon, \quad \text{for some } \varepsilon \in (0, 1].
\end{equation}
This strict ellipticity guarantees that the dynamic Cord\`{e}s multiplier $\lambda(\bs x) = \text{tr}(A^{(k)}) / \text{tr}((A^{(k)})^2)$ is well-defined, strictly positive, and bounded away from zero. Thus, it theoretically validates the contraction mapping of the preconditioned operator at every Newton iteration step.

\subsection{Example 5.1 Square-to-Square Optimal Transport} To demonstrate the practical implementation of the proposed method in optimal transport, we consider a canonical example involving transport between two square domains. This setting not only facilitates visualization of the transport map but also serves as a standard benchmark for assessing numerical performance.

We first consider a benchmark example with an analytical solution taken from \cite{Benamou2014}. Let us introduce the following auxiliary function:
\begin{equation}
    q(z)=\left( -\frac{1}{8\pi}z^2+\frac{1}{256\pi ^3}+\frac{1}{32\pi} \right) \cos(8\pi z)+\frac{1}{32\pi ^2}z\sin(8\pi z).
\end{equation}
Based on this function, the two-dimensional source density is defined as
\begin{equation}
    \begin{aligned}
        f(x_1,x_2) = & 1 + 4 \left( q''(x_1)q(x_2) + q(x_1)q''(x_2) \right) \\
        &+ 16 \left( q(x_1)q(x_2)q''(x_1)q''(x_2) - (q'(x_1))^2(q'(x_2))^2 \right),
    \end{aligned}
\end{equation}
on the square domain $\Omega = (-\frac{1}{2},\frac{1}{2}) \times (-\frac{1}{2},\frac{1}{2}) $. The target density is taken to be the uniform distribution on the same domain.

For this problem, the optimal transport map admits an explicit analytical form and can be represented as the gradient of a convex potential $u$, given by
\begin{equation}
    \begin{aligned}
        u_{x_1}(x_1,x_2)=x_1+4q'(x_1)q(x_2), \\
        u_{x_2}(x_1,x_2)=x_2+4q(x_1)q'(x_2).
    \end{aligned}
\end{equation}

In Fig \ref{fig:ma2}, we present the Cartesian grid on the source domain together with the corresponding source density (left), and the deformed grid obtained under the optimal transport map (right). It can be observed that the source density exhibits a pronounced peak near the center of the domain, while remaining relatively flat near the boundary.
\vspace{-0.05in}
\begin{figure}[H]
    \centering
    \includegraphics[width=0.8\linewidth]{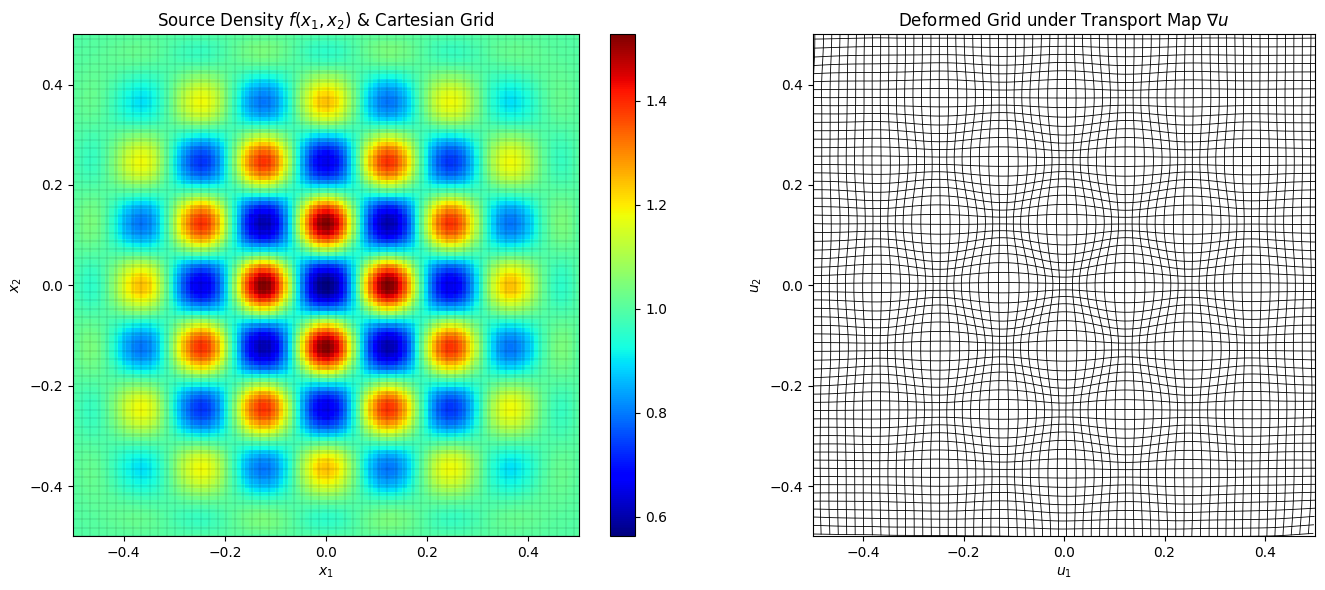}
    \caption{Source density and transported grid under the optimal transport map}
    \label{fig:ma2}
\end{figure}
\vspace{-0.05in}
Under the action of the optimal transport map, the grid undergoes a significant redistribution: the initially uniform Cartesian grid becomes denser in low-density regions and sparser in high-density regions. This behavior provides a clear illustration of the mass conservation mechanism, whereby the transport map adjusts local volume elements, compressing regions of low density and expanding those of high density, to match the source and target distributions.

\subsection{Example 5.2 Optimal Transport on Complex Geometries}
To further demonstrate the robustness and scalability of the proposed method, we extend our study to optimal transport problems defined on geometrically complex surfaces. In particular, we consider a human cortical surface and a lion head model in Fig \ref{fig:ot1}, both characterized by intricate geometrical structures and rich local details \cite{Lei2021}. These examples serve as stringent tests for evaluating the performance of the method in realistic and high-complexity scenarios.

\begin{figure}[H]
    \centering
    \includegraphics[width=0.8\linewidth]{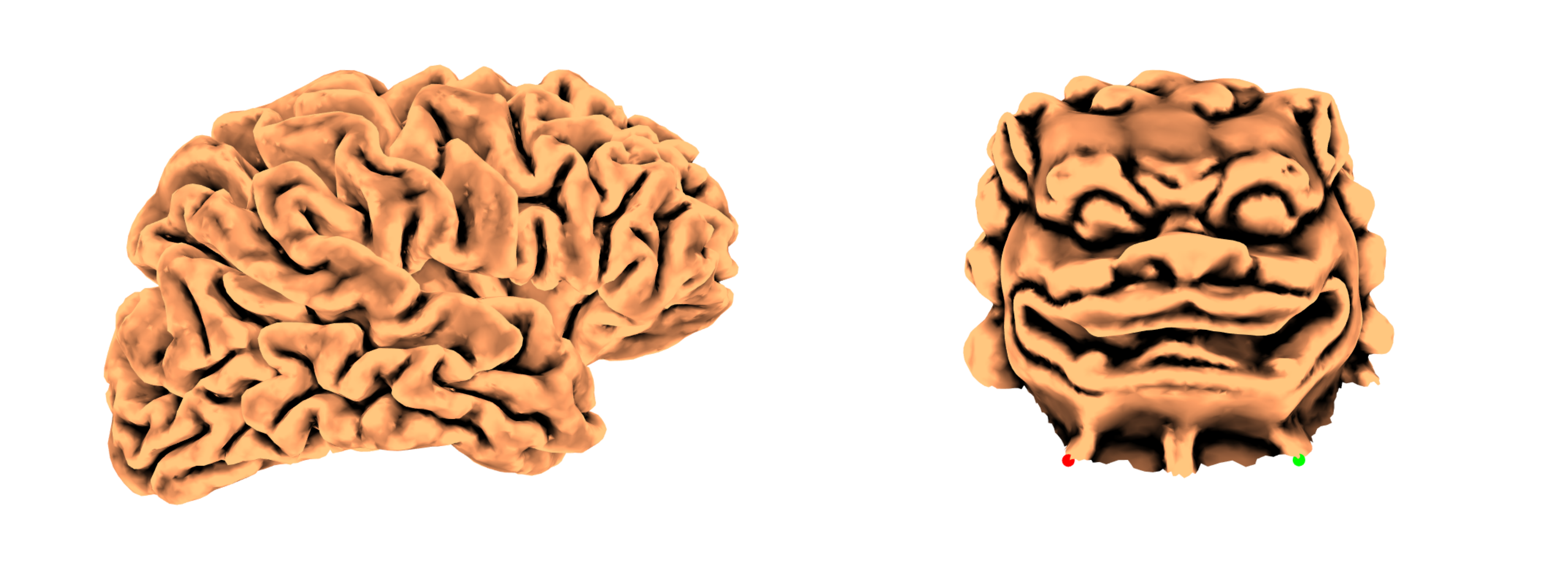}
    \caption{3D surface models of the cortical surface and the lion head}
    \label{fig:ot1}
\end{figure}

The method in \cite{Meng2016} maps the 3D surface onto a 2D planar domain via conformal parameterization. Numerical results show that this approach is stable in terms of angle preservation and effectively maintains the local directional structures of cortical folds as well as the global topological features. However, from the perspective of area distribution in the parameter domain, the induced conformal factor exhibits significant spatial non-uniformity. Regions with high curvature and dense folding are severely compressed, while relatively flat regions with low curvature are noticeably stretched.

Based on this observation, we define the conformal factor as the source density $f$, and take the target density $g$ as a constant. By solving the corresponding optimal transport problem, we obtain a mass-preserving map that redistributes points in the parameter domain, expanding high-density regions and compressing low-density ones, thereby correcting the area distortion introduced by the conformal parameterization. In Fig \ref{fig:ot2}, the left column shows the mapping obtained by the method in \cite{Meng2016}, while the right column presents the results produced by the proposed method.

\begin{figure}[H]
    \centering
    \includegraphics[width=0.4\linewidth]{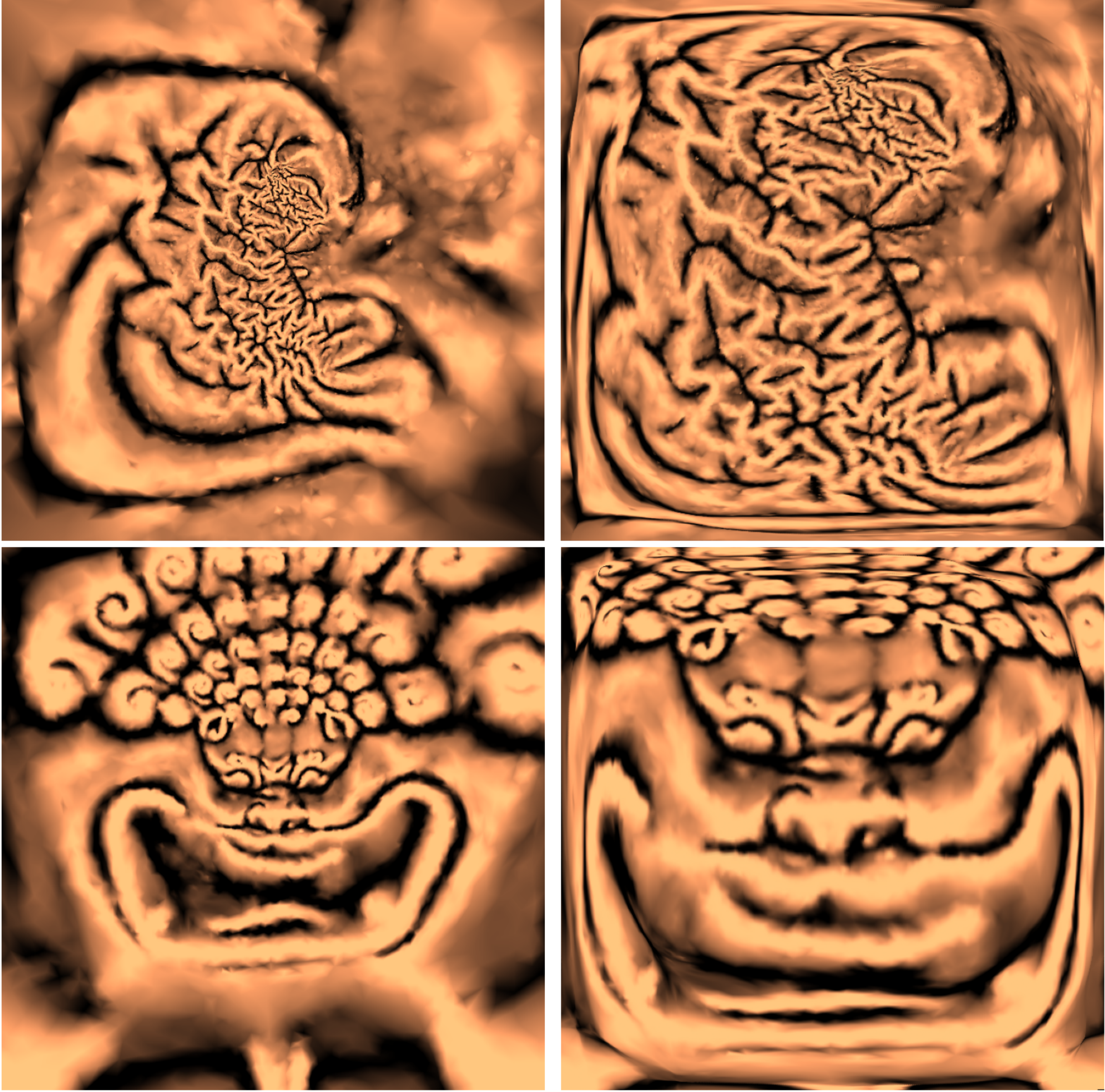}
    \caption{Comparison of conformal and optimal transport mappings}
    \label{fig:ot2}
\end{figure}
\section{Conclusion}
In this work, we propose an improved PINN framework based on the Cord\`{e}s condition, demonstrating its effectiveness in solving both linear and fully nonlinear PDEs. The proposed C-PINN enhances the loss formulation and provides a stable and accurate numerical framework for problems in non-divergence form and beyond. The key contributions of this work can be summarized as follows:
\begin{enumerate}
    \item Introducing a Cord\`{e}s condition-based modification to the PINN loss function, enabling stable and accurate numerical solutions of linear elliptic equations in non-divergence form.
    \item Extending the proposed framework to Hamilton-Jacobi-Bellman equations, demonstrating its capability in handling fully nonlinear PDEs.
    \item Further applying the method to the Monge-Amp\`{e}re equation via linearization and highlighting the generality and flexibility of the C-PINN framework.
    \item Validating the proposed approach through optimal transport applications and demonstrating its effectiveness in practical and geometrically complex scenarios.
\end{enumerate}
Ultimately, the overarching advantage of the proposed C-PINN methodology lies in its seamless integration of rigorous classical PDE theory with modern deep learning optimization. Unlike standard PINNs that frequently suffer from severe gradient pathologies and non-convex loss landscapes when confronted with unconditioned high-order derivatives, C-PINN mathematically preconditions the residual space. This transformation mathematically guarantees a strict contraction mapping property, fundamentally eliminating the training instability inherent to non-divergence and fully nonlinear operators. Consequently, our approach not only preserves the intrinsic mesh-free flexibility and high-dimensional scalability of neural networks, but also endows them with the robust theoretical reliability of traditional numerical analysis, paving a highly efficient and mathematically grounded pathway for complex physical and geometrical applications. Owing to its generality and simplicity, the proposed method is expected to be of broad interest to the scientific and engineering communities.
\\

\appendix

\section{Analytical Construction of Distance Functions}
\begin{table}[htbp]
    \centering
    \caption{Summary of the distance functions $\mathcal{D}(\boldsymbol{x})$ utilized to enforce hard boundary constraints across different test problems.}
    \label{tab:appendix_distance}

    \begin{tabular}{l c}
    \toprule
    Domain $\Omega$ & Distance Function $\mathcal{D}(\boldsymbol{x})$ \\
    \midrule

    $(-2, 2) \times (-2, 2)$ & $(4-x_1^2)(4-x_2^2)$  \\
    \addlinespace

    $\{\boldsymbol{x} \in \mathbb{R}^2 \mid x_1^2 + x_2^2 < 4\}$ & $4 - x_1^2 - x_2^2$  \\
    \addlinespace

    $(-1, 1) \times (-1, 1)$ & $(1-x_1^2)(1-x_2^2)$  \\
    \addlinespace

    $\{\boldsymbol{x} \in \mathbb{R}^3 \mid \sum_{i=1}^3 (x_i/a_i)^2 < 1\}$ & $1 - \sum_{i=1}^3 (x_i/a_i)^2$\\
    \addlinespace

    $(-1,1)^5$ & $\prod_{i=1}^5 (1 - x_i^2)$\\
    \addlinespace

    $(-1/2, 1/2) \times (-1/2, 1/2)$ & $(1/4-x_1^2)(1/4-x_2^2)$  \\
    \addlinespace

    $(0, 1) \times (0, 1)$ & $x_1(1-x_1)x_2(1-x_2)$  \\
    \addlinespace

    \bottomrule
    \end{tabular}
\end{table}

\noindent {\bf Data Availability} The datasets generated during and/or analysed during the current study are available from the corresponding author on request.
\vskip 0.1in
\noindent {\bf Declarations}\\
\noindent {\bf Conflict of interest} The authors declared that they have no conflict of interest.

\bibliographystyle{plain}
\end{document}